\documentclass[11pt, a4paper]{amsart}
\usepackage{amssymb, array, amsmath, amscd, pdfpages, enumerate, amsthm, fixltx2e, setspace, hyperref,mathtools}

\usepackage[utf8]{inputenc}
\usepackage{amssymb}
\usepackage{bm}
\usepackage{graphicx}

\newcommand{\eps}{\varepsilon}

\DeclareMathOperator{\Span}{span}
\DeclareMathOperator{\diag}{diag}

\newcommand*{\C}{{\mathbb{C}}}     
\newcommand*{\R}{{\mathbb{R}}}

\newcommand*{\N}{{\mathbb{N}}}

\newcommand*{\Lin}{{\mathcal{L}}}   
\newcommand*{\Dom}{{\mathcal{D}}}   
\newcommand{\ran}{{\mathcal{R}}}   
\renewcommand{\ker}{{\mathcal{N}}}

\newcommand*{\abs} [1]{\lvert#1\rvert}
\newcommand*{\norm}[1]{\lVert#1\rVert}
\newcommand*{\set} [1]{\{#1\}}
\newcommand*{\setm}[2]{\{\,#1\mid#2\,\}}   
\newcommand*{\iprod}[2]{\langle#1,#2\rangle}

\newcommand*{\Lp}[1][p]{L^{#1}}
 
\newcommand*{\Xloc}[1]{#1_{\text{loc}}} 
\newcommand*{\Lploc}[1][p]{\Xloc{L^{#1}}}

\newcommand{\pmat}[1]{\begin{pmatrix}#1\end{pmatrix}}
\newcommand{\pmatsmall}[1]{\begin{psmallmatrix}#1\end{psmallmatrix}}

\newcommand*{\Abs}[2][default]{\ifthenelse{\equal{#1}{default}}{\left\lvert#2\right\rvert}{\ldelim{#1}{\lvert}#2\rdelim{#1}{\rvert}}}
\newcommand*{\Norm}[2][default]{\ifthenelse{\equal{#1}{default}}{\left\lVert#2\right\rVert}{\ldelim{#1}{\lVert}#2\rdelim{#1}{\rVert}}}

\newcommand*{\Iprod}[3][default]{\ifthenelse{\equal{#1}{default}}{\left\langle#2,#3\right\rangle}{\ldelim{#1}{\langle}#2,#3\rdelim{#1}{\rangle}}}
\newcommand*{\Dualpair}[3][default]{\ifthenelse{\equal{#1}{default}}{\left\langle#2,#3\right\rangle}{\ldelim{#1}{\langle}#2,#3\rdelim{#1}{\rangle}}}

\newcommand*{\List}[2][1]{\set{#1,\ldots,#2}}

\newcommand{\eq}[1]{\begin{align*}#1\end{align*}}
\newcommand{\eqn}[1]{\begin{align}#1\end{align}}

\newcommand{\gs}{\sigma}
\newcommand{\ga}{\alpha}

\newcommand{\gd}{\delta}
\newcommand{\gl}{\lambda}
\newcommand{\gw}{\omega}

\newcommand{\ieq}[1]{$#1$}

\newcommand{\inv}{^{-1}}

\newcommand*{\ddb}[2][1]{\ifthenelse{\equal{#1}{1}}{\frac{d}{d#2}}{\frac{d^{#1}}{d#2^{#1}}}}
\newcommand*{\pd}[3][1]{\ifthenelse{\equal{#1}{1}}{\frac{\partial{#2}}{\partial{#3}}}{\frac{\partial^{#1}{#2}}{\partial#3^{#1}}}}

\newcommand*{\keyterm}[1]{\emph{#1}}

\newcommand{\Pop}{\mathcal{P}}
\newcommand{\Pdop}{\mathcal{P}_d}
\newcommand{\Poppert}{\tilde{\mathcal{P}}}
\newcommand{\Pdoppert}{\tilde{\mathcal{P}}_d}

\newcommand{\YP}{Q_N}
\newcommand{\FBop}{K_s}

\newcommand{\KA}{G_A}

\newtheorem{thm}{Theorem}[section]
\newtheorem{lem}[thm]{Lemma}
\newtheorem{cor}[thm]{Corollary}
\theoremstyle{definition}
\newtheorem{dfn}[thm]{Definition}

\newtheorem{rem}[thm]{Remark}

\newtheorem*{ORPff}{The Feedforward Output Regulation Problem}
\newtheorem*{ORPfb}{The Error Feedback Output Regulation Problem}
\newtheorem*{RORP}{The Robust Output Regulation Problem}

\newcommand{\yref}{y_{\mbox{\scriptsize\textit{ref}}}}
\newcommand{\wdist}{w_{\mbox{\scriptsize\textit{dist}}}}
\newcommand{\yrefpert}{\tilde{y}_{\mbox{\scriptsize\textit{ref}}}}
\newcommand{\wdistpert}{\tilde{w}_{\mbox{\scriptsize\textit{dist}}}}
\newcommand{\ydk}[1][k]{y_{d,#1}}
\newcommand{\vdvec}{v_{\mbox{\scriptsize\textit{cf}}}}
\newcommand{\vdvecdef}{\vdvec = (v_1,\ldots,v_q)^T}

\newcommand{\pinv}{^{\dagger}}

\newcommand{\FM}{U }

\newcommand{\Lted}[1]{{\bf#1}}
\newcommand{\Ltedop}[1]{\mathbb{#1}}
\newcommand{\persym}{\tau}

\newcommand{\sysops}{(A,B,C,D,E,F)}
\newcommand{\sysopspert}{(\tilde{A},\tilde{B},\tilde{C},\tilde{D},\tilde{E},\tilde{F})}

\newcommand{\syspar}{(A(t),B(t),B_d(t),C(t),D(t))}
\newcommand{\sysparpert}{(\tilde{A}(t),\tilde{B}(t),\tilde{B}_d(t),\tilde{C}(t),\tilde{D}(t))}

\newcommand{\yrefn}{y^{\mbox{\scriptsize\textit{ref}}}_{n}}
\newcommand{\yrefnL}{\Lted{y}^{\mbox{\scriptsize\textit{ref}}}_{n}}
\newcommand{\yrefL}{\Lted{y}_{\mbox{\scriptsize\textit{ref}}}}

\newcommand{\wdistnL}{\Lted{w}_{n}}

\newcommand{\FFfun}{u_{reg}}
\newcommand{\IEQsol}{u}

\renewcommand{\pmat}[1]{\begin{bmatrix}#1\end{bmatrix}}
\renewcommand{\pmatsmall}[1]{\begin{bsmallmatrix}#1\end{bsmallmatrix}}

\begin{document}

\title{Robust Output Regulation for Continuous-Time Periodic Systems}

\thispagestyle{plain}

\author{Lassi Paunonen}
\address{Department of Mathematics, Tampere University of Technology, PO.\ Box 553, 33101 Tampere, Finland}
\email{lassi.paunonen@tut.fi}
\thanks{The research is funded by the Academy of Finland grant number 298182.}

\begin{abstract}
We consider controller design for robust output tracking and disturbance rejection for continuous-time periodic linear systems with periodic reference and disturbance signals. As our main results we present four different controllers: A feedforward control law and a discrete-time dynamic error feedback controller for output tracking and disturbance rejection, a robust discrete-time feedback controller, and finally a discrete-time feedback controller that achieves approximate robust output tracking and disturbance rejection. The presented constructions are also new for time-invariant finite and infinite-dimensional systems. The results are illustrated with two examples: A periodically time-dependent system of harmonic oscillators and a nonautonomous two-dimensional heat equation with boundary disturbance. 
\end{abstract}


\maketitle

\section{Introduction}
\label{sec:intro}

In this paper we study the \keyterm{output regulation problem} for an exponentially stable $\persym$-periodic system of the form
\begin{subequations}
  \label{eq:plantintro}
  \eqn{
  \label{eq:plantintro1}
  \hspace{-1ex}\dot{x}(t)&= A(t)x(t)+B(t)u(t)+B_d(t)\wdist(t) , ~ x(0)=x_0\\
  \hspace{-1ex}y(t)&=C(t)x(t)+D(t)u(t)
  }
\end{subequations}
with state space $X$, input space $U_0$, and output space $Y_0$.
The main goal in our control problem is to design a control law in such a way that the output $y(t)\in Y_0$ converges asymptotically to a $\tau$-periodic reference signal $\yref(\cdot)$ despite the external $\tau$-periodic reference signal $\wdist(\cdot)$.
In the \keyterm{robust output regulation problem} we in addition require that the same controller achieves the output tracking even for perturbed parameters 
$(\tilde{A}(t),\tilde{B}(t),\tilde{B}_d(t),\tilde{C}(t),$ $\tilde{D}(t))$
%$\sysparpert$
of the system~\eqref{eq:plantintro}.
Throughout the paper we consider systems~\eqref{eq:plantintro} on a Banach or Hilbert space $X$. This class of systems includes a wide range of nonautonomous partial differential equations, delay equations, and infinite systems of ordinary differential equations. However, the presented results are also new and directly applicable for finite-dimensional periodic systems on $X=\C^n$ and $X=\R^n$.
Finally, our results offer three new controllers for output regulation of linear time-invariant systems with nonsmooth periodic reference and disturbance signals.

We consider two types of control laws. The first type is a static $\tau$-periodic control law 
\eqn{
\label{eq:uffintro}
u\in
\Lploc[2](0,\infty;U_0),
\qquad u(\cdot)=\FFfun(\cdot) \quad \mbox{on} \quad [0,\persym],
}
where the function $\FFfun\in\Lp[2](0,\persym;U_0)$ is computed based on $\yref(\cdot)$ and $\wdist(\cdot)$.
The second type is a dynamic feedback control law 
\begin{subequations}
  \label{eq:fbcontrintro}
  \eqn{
  \label{eq:ufbintro}
  u(\cdot)=Kz_n \quad \mbox{on} \quad [n\persym,(n+1)\persym), \quad n\geq 0 \hspace{-1ex}
  }
  where 
  $K\in\Lin(Z, \Lp[2](0,\persym;U_0))$
  and where $z_n$ is the state of the discrete-time controller
  \eqn{
  \label{eq:udynintro}
  z_{n+1} &= G_1 z_n + G_2(y( n\persym+\cdot)-\yref(\cdot)), \quad z_0\in Z,  
  }
\end{subequations}
with $G_1\in \Lin(Z)$ and $G_2 \in \Lin(\Lp[2](0,\persym;Y_0),Z)$.
In this control configuration the control input $u(\cdot)$ on the interval $[n\persym,(n+1)\persym)$ for $n\geq 0$ is determined by $z_n$, which for $n\geq 1$ in turn depends on $z_{n-1}$ and the output $y(\cdot)$  on  $[(n-1)\persym,n\persym)$.

In this paper we present four different controllers. The first controller is a static $\persym$-periodic control law of the form~\eqref{eq:uffintro} that solves the output regulation problem in the situation where both the reference signal $\yref(\cdot)$ and the disturbance signal $\wdist(\cdot)$ are known functions. The second controller is a finite-dimensional discrete-time feedback controller~\eqref{eq:fbcontrintro} that achieves output regulation for any disturbance signal $\wdist(\cdot)$ that is a linear combination of a finite number of known $\tau$-periodic functions. This controller can in particular be used when the frequencies of the disturbance signal $\wdist(\cdot)$ are known, but the amplitudes and the phases are unknown.
 
The last two controllers presented in the paper are designed to solve the robust output regulation problem. Our third controller is a robust discrete-time feedback controller that achieves output tracking and disturbance rejection even under perturbations and uncertainties in the parameters of the system~\eqref{eq:plantintro}. We will see that the internal model principle~\cite{FraWon75a,Dav76,PauPoh10} implies that in order to tolerate arbitrary small perturbations in the system~\eqref{eq:plantintro}, the controller~\eqref{eq:fbcontrintro} must necessarily be infinite-dimensional. However, we will also show that if the goal of the asymptotic output tracking is relaxed to \keyterm{approximate} convergence of the output $y(t)$ to the reference signal $\yref(\cdot)$, then the robust output regulation problem can be solved with a finite-dimensional controller. In particular, the fourth and final controller we present is a finite-dimensional discrete-time feedback controller that achieves approximate output tracking 
in the sense that the regulation error becomes small as $t\to \infty$,
and is robust with respect to small perturbations in the parameters of the system. 
 
The constructions of the controllers are completed using two operators 
$\Pop\in \Lin(\Lp[2](0,\persym;U_0),\Lp[2](0,\persym;Y_0))$ and  $\Pdop\in \Lin(\Lp[2](0,\persym;U_{d0}),\Lp[2](0,\persym;Y_0))$ associated to the periodic system~\eqref{eq:plantintro}.
If we denote by $U_A(t,s)$, $t\geq s$, the strongly continuous evolution family associated to~\eqref{eq:plantintro1}~\cite{EngNag00book}, then the operators $\Pop$ and $\Pdop$ are
defined in such a way that for all $u\in \Lp[2](0,\persym;U_0)$ and $w\in \Lp[2](0,\persym;U_{d0})$
\eq{
\Pop u &=  C(\cdot)\int_0^\persym \hspace{-.8ex}\KA(\cdot,s)B(s)u(s)ds + C(\cdot)\int_0^\cdot \hspace{-.5ex}\FM_A(\cdot,s)B(s)u(s)ds + D(\cdot) u(\cdot) 
}
\eq{
\Pdop w &= C(\cdot)\int_0^\persym \hspace{-.8ex}\KA(\cdot,s)B_d(s)w(s)ds + C(\cdot)\int_0^\cdot \hspace{-.5ex}\FM_A(\cdot,s)B_d(s)w(s)ds
}
where 
  \eq{
  \KA(t,s) &= \FM_A(t,0)(I-\FM_A(\persym,0))\inv \FM_A(\persym,s).
  }
The operator $I-\FM_A(\persym,0)$ is boundedly invertible since~\eqref{eq:plantintro} is exponentially stable.  
The operators $\Pop$ and $\Pdop$ describe the steady state output of the stable periodic system~\eqref{eq:plantintro} under $\persym$-periodic inputs $u(\cdot)$ and disturbances $\wdist(\cdot)$, respectively. 
In particular, we will show that if $u(\cdot)\in \Lploc[2](\R,U_0)$ and $\wdist(\cdot)\in \Lploc[2](\R,U_{d0})$ are periodic extensions of the functions $u_0(\cdot)$ and $w_0(\cdot)$ on $[0,\persym]$, then the output $y(t)$ of~\eqref{eq:plantintro} converges to the $\persym$-periodic extension of the function $y_0=\Pop u_0 + \Pdop w_0$ in the sense that 
\eq{
\norm{y(n\persym+\cdot)-y_0(\cdot)}_{\Lp[2](0,\persym)}\to 0 \qquad \mbox{as} \quad n\to \infty.
}
The choices of the controller parameters are based on solutions of linear equations of the form
\ieq{
y_0 = \Pop u_0
}
for certain $y_0\in \Lp[2](0,\persym;Y_0)$. The form of the operator $\Pop$ implies that $u_0\in \Lp[2](0,\persym;U_0)$ can be obtained as a solution of a Volterra--Fredholm integral equation. 
We will further show that for a stable periodic system the operator $\Pop$ and the function $u_0$ can be approximated based on measurements taken from the system~\eqref{eq:plantintro} using a 
straightforward procedure introduced in Section~\ref{sec:P1measurement}.

Output regulation of finite-dimensional nonautonomous systems has been studied in several references~\cite{Pin93,MarTom00,IchKat06,ZheSer06,ZonZhe09,ZheZon09}. Moreover, robust controllers based on time-dependent internal models have been introduced in~\cite{ZheSer09}.
For infinite-dimensional periodic systems the output regulation problem was studied in~\cite{PauPoh12b} for an autonomous system and reference and disturbance signals generated by a periodic exosystem. In this paper we employ the so-called \keyterm{lifting technique}~\cite{MeyBur75,BamPea91} to introduce novel controllers for finite and infinite-dimensional periodic systems. Lifting has been successfully used in the study of robust output regulation for periodic discrete-time systems in~\cite{GraLon91a,GraLon91b,GraLon96,JemDav03,NagYam09}. However, extending the lifting approach to continuous-time systems poses many mathematical challenges due to the infinite-dimensional input and output spaces of the resulting lifted systems. In this paper we demonstrate that the lifting approach remains a powerful tool also in controller design for continuous-time periodic systems. In particular, controller design for lifted system leads naturally to discrete-time dynamic error feedback controllers of the form~\eqref{eq:fbcontrintro}.

The lifting approach for robust output regulation of continuous-time systems was first used in~\cite{PauCT15} for finite-dimensional periodic systems without disturbance rejection. The third controller presented in this paper generalizes the controller in~\cite{PauCT15} to infinite-dimensional systems with external disturbance signals, and weakens the assumptions required in the construction. The first controller in this paper extends the feedforward control law originally presented in~\cite{PauPoh12b} for autonomous systems with periodic exosystems. We show that the presented construction of the feedforward control law using the operators $\Pop$ and $\Pdop$ is equivalent to the solution of the \keyterm{periodic regulator equations} consisting of an infinite-dimensional Sylvester differential equation and a regulation constraint. The two discrete-time error feedback controllers that are presented for output regulation and for approximate robust output regulation are completely new. 

We illustrate the theoretic results with two examples. In the first example we consider a system consisting of two harmonic oscillators with periodic damping and periodic coupling. In the second example we design controllers for output tracking and robust output tracking for a periodically time-dependent two-dimensional heat equation with boundary disturbances.

The paper is organized as follows. In Section~\ref{sec:objectivesandass} we state the standing assumptions on the system~\eqref{eq:plantintro} and formulate the main control problems. 
The constructions of all the controllers are presented in Section~\ref{sec:contrdesign}. 
The proofs of the main theorems are presented separately in Section~\ref{sec:proofsmain}.
In Section~\ref{sec:P1measurement} we present a method for approximating the operators $\Pop$ and $\Pdop$ based on measurements from the system~\eqref{eq:plantintro}.
The examples where we consider controller design for the system of harmonic oscillators and the periodic heat equation are presented in Sections~\ref{sec:HarmOsc} and~\ref{sec:heatsys}, respectively. Section~\ref{sec:conclusions} contains concluding remarks.

If $X$ and $Y$ are Banach spaces and $A:X\rightarrow Y$ is a linear operator, we denote by $\Dom(A)$, $\ker(A)$ and $\ran(A)$ the domain, kernel and range of $A$, respectively. 
The space of bounded linear operators from $X$ to $Y$ is denoted by $\Lin(X,Y)$.
If \mbox{$A:X\rightarrow X$,} then $\gs(A)$ and $\rho(A)$ denote the spectrum and the \mbox{resolvent} set of $A$, respectively. For $\gl\in\rho(A)$ the resolvent operator is  \mbox{$R(\gl,A)=(\gl -A)^{-1}$}. If $X$ and $Y$ are Hilbert spaces then $A^\ast$ is the adjoint of $A\in \Lin(X,Y)$.  The inner product on a Hilbert space is denoted by $\iprod{\cdot}{\cdot}$.
The space of $\persym$-periodic $X$-valued functions is denoted by $C_\persym(\R,X)$.

\section{Standing Assumptions and Control Objectives}
\label{sec:objectivesandass}

We begin by stating the standing assumptions on the system~\eqref{eq:plantintro}.
The parameters $(A(\cdot),B(\cdot),B_d(\cdot),C(\cdot),D(\cdot))$ are operator-valued $\persym$-periodic functions satisfying
$B(\cdot)\in \Lp[\infty](\R,\Lin(U_0,X))$, $B_d(\cdot)\in \Lp[\infty](\R,\Lin(U_{d0},X))$, $C(\cdot)\in \Lp[\infty](\R,\Lin(X,Y_0))$, and $D(\cdot)\in \Lp[\infty](\R,\Lin(U_0,Y_0))$,
where
$U_0$, $U_{d0}$, and $Y_0$ are Hilbert spaces.  
We denote $U=\Lp[2](0,\persym;U_0)$, $U_d=\Lp[2](0,\persym;U_{d0})$, and $Y = \Lp[2](0,\persym;Y_0)$.
We assume 
there exists a strongly continuous evolution family $U_A(t,s)$~\cite{EngNag00book} 
satisfying $\FM_A(t,t)=I$ and $\FM_A(t,r)\FM_A(r,s)=\FM_A(t,s)$ for all $t\geq r\geq s$ such that 
for all $u\in\Lploc[1](\R;U_0)$ the system~\eqref{eq:plantintro} has a well-defined mild state given by
\eq{
x(t) &= \FM_A(t,0)x_0 + \int_0^t \FM_A(t,s)B(s)u(s)ds.
}
If space $X$ is finite-dimensional, then the evolution family $U_A(t,s)$ is given by the fundamental matrix of the ordinary differential equation~\eqref{eq:plantintro1}. More generally, the assumption is in particular true if $A(\cdot)=A_0+A_1(\cdot)$ where $A_0: \Dom(A_0) \subset X\to X$ generates a strongly continuous semigroup $T_0(t)$ on $X$ and 
$A_1\in \Lp[\infty](\R, \Lin(X))$ 
is $\persym$-periodic. In this case the strongly continuous evolution family $\FM_A(t,s)$ is uniquely determined by 
the integral equations
\eq{
\FM_A(t,s)x = T_0(t-s)x + \int_s^t T_0(t-r)A_1(r)\FM_A(r,s)xdr.
}
The perturbed systems considered in the robust output regulation problem are assumed to satisfy the same standing assumptions as the nominal system~\eqref{eq:plantintro}.
The fact that $A(\cdot)$ is $\persym$-periodic implies that $\FM_A(t+\persym,s+\persym)=\FM_A(t,s)$ for all $t\geq s$.
In this paper we study the control of stable systems, and we therefore assume that the evolution family $\FM_A(t,s)$ is  exponentially stable, i.e., there exist $M,\gw>0$ such that $\norm{\FM_A(t,s)}\leq Me^{-\gw(t-s)}$ for all $t\geq s$. 
The following characterization of exponential stability of a periodic evolution family $\FM_A(t,s)$ follows from the property $\FM_A(n\persym,0)=\FM_A(\persym,0)^n$ for all $n\in\N$ and~\cite[Prop. II.1.3]{Eis10book}.

\begin{lem}
  \label{lem:UAexpstabchar}
If $A(\cdot)$ is $\persym$-periodic, then
$\FM_A(t,s)$ is exponentially stable if and only if
$\abs{\gl}<1$ for all $\gl\in \gs(\FM_A(\persym,0))$.  
\end{lem}

\begin{rem}
  \label{rem:unbddIO}
  Even though we assumed that the values of the functions $B(\cdot)$, $B_d(\cdot)$, and $C(\cdot)$  are bounded linear operators, 
  the results in this paper remain valid also for certain classes of systems where $B(t)$, $B_d(t)$, and $C(t)$ are unbounded operators~\cite{Sch02}.
  In particular, it is sufficient to pose conditions under which the lifted system in Section~\ref{sec:liftedsys} is well-defined.
  This requirement is in particular satisfied if the system~\eqref{eq:plantintro} 
  is a time-invariant 
  \keyterm{regular linear system}~\cite{Wei94}.
\end{rem}

\begin{rem}
The results in this paper can also be used for constructing controllers for unstable systems if the system~\eqref{eq:plantintro} can first be stabilized with either state feedback $u(t) = \FBop(t)x(t) + \tilde{u}(t)$ 
with a $\persym$-periodic $\FBop(\cdot)\in \Lp[\infty](\R,\Lin(X,U_0))$ (if $x(t)$ is available for feedback) or with
output feedback $u(t) = \FBop(t)y(t) + \tilde{u}(t)$ with a $\persym$-periodic $\FBop(\cdot)\in \Lp[\infty](\R,\Lin(Y_0,U_0))$ such that $(I-\FBop(\cdot)D(\cdot))\inv \in \Lp[\infty](\R,\Lin(U_0))$.
The controllers can then be designed for the stabilized system with the new input $\tilde{u}(t)$ provided that the stabilized system has a well-defined mild state given by a strongly continuous evolution family. 
For infinite-dimensional systems sufficient conditions for this property are presented, e.g., in~\cite{HinPri94,Sch02}, and~\cite[Sec. VI.9.c]{EngNag00book}.
\end{rem}

Throughout the paper we assume that the reference and disturbance signals are $\persym$-periodic functions such that 
\eqn{
\label{eq:wdistform}
\yref(\cdot)\in \Lploc[2](0,\infty;Y_0) , 
\quad \mbox{and} \quad
\wdist(\cdot)&= \sum_{k=1}^q v_k \wdist^k(\cdot)
} 
for some $\persym$-periodic functions $\wdist^k(\cdot)\in \Lploc[2](0,\infty;U_{d0})$ and some unknown coefficient vector $\vdvecdef\in \C^q$.

\subsection{Control Objectives}
\label{sec:controlobjectives}

The main goal in all of the control problems is to achieve the convergence of $y(\cdot)$ to a $\persym$-periodic reference signal $\yref(\cdot)$ in the sense that
for all initial states of the system and the controller the integrals
\eq{
\int_{n\tau}^{(n+1)\tau} \norm{y(t)-\yref(t)}^2 dt 
}
decay to zero at a uniform exponential rate as $n\to \infty$.
This form of convergence differs from the pointwise convergence where $\norm{y(t)-\yref(t)}_{Y_0}\to 0$ at an exponential rate as $t\to\infty$, but we will see that it is a natural choice to use in connection with the lifting approach used in this paper.
The three main control problems are defined in the following.

\begin{ORPff}
  For given fixed signals $\yref(\cdot)$ and $\wdist(\cdot)$ choose a control input $u(\cdot)\in\Lploc[2](0,\infty;U_0)$ in such a way that for some $M,\ga>0$ we have
      \eq{
      \int_{n\tau}^{(n+1)\tau} \norm{y(t)-\yref(t)}^2 dt 
      \leq M^2 e^{-2\ga n}\left( \norm{x_0}^2   + 1 \right) 
      }
      for all $x_0\in X$ and $n\geq 0$.
\end{ORPff}

The next control objective considers output tracking and disturbance rejection using a feedback controller of the form~\eqref{eq:fbcontrintro}. In this control problem we can consider disturbance signals of the form~\eqref{eq:wdistform} with unknown coefficients $\set{v_k}_{k=1}^q\subset \C$.
In our context the \keyterm{exponential closed-loop stability} means that there exist $M_0,M_1,\ga_0,\ga_1>0$ such that in the case where $\yref(\cdot)\equiv 0$ and $\wdist(\cdot)\equiv 0$ we have
  \eq{
  \norm{x(t)} &\leq M_0e^{-\ga_0 t}(\norm{x_0}+\norm{z_0} ), \\
  \norm{z_n} &\leq M_1e^{-\ga_1 n}(\norm{x_0}+\norm{z_0} )
  }
  for all $t\geq 0$ and $n\geq 0$ and for all initial states $x_0\in X$ and $z_0\in Z$.

\begin{ORPfb}
  Choose the parameters $(G_1,G_2,K)$ of the dynamic feedback controller~\eqref{eq:fbcontrintro} in such a way that 
  \begin{itemize}
      \setlength{\itemsep}{1ex}
    \item[\textup{(1)}] The closed-loop system is exponentially stable.
    \item[\textup{(2)}]
       The output converges to the reference signal in the sense that for some $M,\ga>0$ and for all initial states $x_0\in X$ and $z_0\in Z$ 
and for all $\vdvecdef \in \C^q$
      \eq{
      \int_{n\tau}^{(n+1)\tau}
      \hspace{-1ex}
      \norm{y(t)-\yref(t)}^2 dt 
\leq M^2 e^{-2\ga n} \left( \norm{x_0}^2 + \norm{z_0}^2 +  \norm{\vdvec}^2 + 1 \right)
      }
      for all $n\geq 0$.
  \end{itemize}
\end{ORPfb}

Finally, in the robust output regulation problem it is in addition required that the error feedback controller tolerates perturbations and uncertainties in the parameters $\syspar$ of the system~\eqref{eq:plantintro}. The robustness of the controller also implies that the controller is capable of tracking \keyterm{any} $\persym$-periodic reference signal $\yrefpert(\cdot)\in\Lploc[2](0,\infty;Y_0)$ and rejecting any $\persym$-periodic disturbance signal $\wdistpert(\cdot)\in\Lploc[2](0,\infty;U_{d0})$

\begin{RORP}
  Choose  $(G_1,G_2,K)$ in the dynamic feedback controller~\eqref{eq:fbcontrintro} in such a way that 
  \begin{itemize}
      \setlength{\itemsep}{1ex}
    \item[\textup{(1)}] The closed-loop system is exponentially stable.
    \item[\textup{(2)}]
       The output converges to the reference signal in the sense that for some $M,\ga>0$ 
       and for all initial states $x_0\in X$ and $z_0\in Z$ 
and for all $\vdvecdef \in \C^q$
      \eq{
       \int_{n\tau}^{(n+1)\tau} \hspace{-1ex} \norm{y(t)-\yref(t)}^2 dt 
\leq M^2 e^{-2\ga n} \left( \norm{x_0}^2 + \norm{z_0}^2 +  \norm{\vdvec}^2 +1 \right)
      }
      for all $n\geq 0$.
    \item[\textup{(3)}] If $\yref(\cdot)$ and $\wdist(\cdot)$ are changed to $\persym$-periodic signals $\yrefpert(\cdot)$ and $\wdistpert(\cdot)$ and if
      the parameters $\syspar$ are perturbed to $\sysparpert$ in such a way that the exponential closed-loop stability is preserved, then the property~(2) continues to hold for some constants $M,\ga>0$.
  \end{itemize}
\end{RORP}

\section{Construction of The Controllers} 
\label{sec:contrdesign}

In this section we present our main results on the construction of controllers. The proofs of all the theorems are presented later in Section~\ref{sec:proofsmain}.

\subsection{Feedforward Output Regulation}
\label{sec:FFcontroller}

The following theorem presents a periodic control law that achieves output tracking of a given reference signal $\yref(\cdot)$ and rejects the known disturbance signal $\wdist(\cdot)$.

\begin{thm}
  \label{thm:ORPstatic}
  Assume the system is exponentially stable. If there exists $\FFfun(\cdot)\in \Lp[2](0,\persym;U_0)$ such that 
\eqn{
\label{eq:ORPstatic}
\Pop\FFfun=\yref-\Pdop\wdist,
}
then the $\persym$-periodic control law $u(\cdot)$ that is the periodic extension of $\FFfun(\cdot)$ from $[0,\persym]$ to $[0,\infty)$ solves the feedforward output regulation problem.

Conversely, if $\FFfun\in \Lploc[2](0,\infty;U_0)$ is a $\persym$-periodic control input such that $\norm{y(n\persym+\cdot)-\yref(\cdot)}_{\Lp[2](0,\persym)}\to 0$ as $n\to \infty$, then $\FFfun$ satisfies~\eqref{eq:ORPstatic} on $[0,\persym]$.
\end{thm}

If instead of a single disturbance signal $\wdist(\cdot)$ we want to reject signals of the form~\eqref{eq:wdistform} with $\set{v_k}_{k=1}^q$, we can achieve this by finding $\set{\FFfun^k}_{k=0}^q\subset  U$ such that
\eq{
\Pop\FFfun^0=\yref(\cdot) 
\qquad \mbox{and}\qquad
\Pop\FFfun^k=\Pdop \wdist^k(\cdot), \quad 
1\leq k\leq q.
}
The linearity of the operators $\Pop$ and $\Pdop$ then implies that for all $\set{v_k}_{k=1}^q$ the function $\FFfun(\cdot)$ in Theorem~\ref{thm:ORPstatic} is given by $\FFfun = \FFfun^0 - \sum_{k=1}^q v_k\FFfun^k$.

\begin{rem}
  \label{rem:FFstrongstab}
  It follows from the proof of Theorem~\ref{thm:ORPstatic} that the same control law also solves the output regulation problem in the case where the system~\eqref{eq:plantintro} is only \keyterm{strongly stable}, meaning that $\norm{U_A(t,0)x_0}\to 0$ as $t\to \infty$ for all $x_0\in X$, and $1\in \rho(U_A(\persym,0))$. In this situation the output $y(t)$ converges to the reference signal in the sense that for all initial states $x_0\in X$ 
  \eq{
      \int_{n\tau}^{(n+1)\tau} \norm{y(t)-\yref(t)}^2 dt \to 0
      \qquad \mbox{as} \quad n\to\infty.
  }
\end{rem}

The form of the operator $\Pop$ implies that finding the solution $\FFfun$ of the equation $\Pop\FFfun = \yref-\Pdop\wdist$ in Theorem~\ref{thm:ORPstatic} is equivalent to finding the solution
$\IEQsol(\cdot)\in U$ of the Volterra--Fredholm integral equation
\eq{
\yref(t) -(\Pdop\wdist)(t)
= D(t) \IEQsol(t) + \int_0^\persym K_F(t,s)\IEQsol(s)ds  + \int_0^t K_V(t,s)\IEQsol(s)ds
}
with kernels
\begin{subequations}
  \label{eq:IEQkernels}
  \eqn{
  \hspace{-1.5ex}  K_F(t,s) &= C(t)\FM_A(t,0)(I-\FM_A(\persym,0))\inv \FM_A(\persym,s)B(s) \\
  \hspace{-1.5ex}K_V(t,s) &= C(t)\FM_A(t,s)B(s).
  }
\end{subequations}
Alternatively, the solution $\FFfun$ of $\Pop\FFfun = \yref-\Pdop\wdist$ can be approximated based on measurements from the periodic system~\eqref{eq:plantintro} using the procedure introduced in Section~\ref{sec:P1measurement}.

The periodic control law $u(\cdot)$ in Theorem~\ref{thm:ORPstatic} can also be characterized as part of the solution of the \keyterm{periodic regulator equations} of the form studied in~\cite{ZheSer06,PauPoh12b}. This connection is described in detail in Section~\ref{sec:PerRegEqns}.

Finally, if~\eqref{eq:plantintro} is a time-invariant system $(A,B,B_d,C,D)$, and if $T(t) $ is the semigroup generated by $A$, then 
  \eq{
  \KA(t,s) &=
  (I-T(\persym))\inv T(t-s+\persym)
}
  and 
the operator $\Pop$ and $\Pdop$ simplify so that
\eq{
&\Pop u =  C\hspace{-.8ex}\int_0^\persym \hspace{-1ex} \KA(\cdot,s)Bu(s)ds + C\hspace{-.8ex}\int_0^\cdot\hspace{-.9ex} T(\cdot-s)Bu(s)ds + D u(\cdot) \\
&\Pdop w = C\hspace{-.8ex}\int_0^\persym \hspace{-1ex}\KA(\cdot,s)B_dw(s)ds + C\hspace{-.8ex}\int_0^\cdot \hspace{-1ex}T(\cdot-s)B_dw(s)ds.
}
For time-invariant systems and for $\tau$-periodic reference and disturbance signals $\yref(\cdot)$ and $\wdist(\cdot)$ it is possible to solve the output regulation problem by solving the regulator equations associated to an infinite-dimensional autonomous exosystem~\cite[Thm. 3.1]{ImmPoh06b}. Also in this case the resulting control input $u(\cdot)$ is $\tau$-periodic, and thus by Theorem~\ref{thm:ORPstatic} it is of the form of the periodic feedforward control law considered in this section in the sense that
it satisfies~\eqref{eq:ORPstatic} on $[0,\persym]$.

\subsection{Error Feedback Output Regulation}
\label{sec:FBcontroller}

In this section we consider reference and disturbance signals of the form~\eqref{eq:wdistform} with unknown coefficients $\set{v_k}_{k=1}^q\subset \C$. Similarly as explained in Section~\ref{sec:FFcontroller} the functions $\FFfun^k$ can be solved from Volterra--Fredholm integral equations with kernels~\eqref{eq:IEQkernels}, or approximated using measurements from the system as shown in Section~\ref{sec:P1measurement}.  
\begin{thm}
  \label{thm:ORPfeedback}
  Assume the system is exponentially stable and the reference and disturbance signals are of the form~\eqref{eq:wdistform} where either $\yref\not\equiv0$ or $\wdist^k\not\equiv 0$ for some $k\in \List{q}$. Assume further that there exist $\FFfun^k(\cdot)\in \Lp[2](0,\persym;U_0)$ for $k\in \List[0]{q}$ such that
\eq{
\Pop\FFfun^0=\yref(\cdot) 
\qquad \mbox{and}\qquad
\Pop\FFfun^k=\Pdop \wdist^k(\cdot), \quad 
1\leq k\leq q
}
and let
$ \set{k_1,\ldots,k_r}\subset \List[0]{q}$
be a set of indices such that 
$\set{\FFfun^{k_j}}_{j=1}^r\subset U$ is a basis of the subspace $\Span \set{\FFfun^0,\FFfun^1,\ldots,\FFfun^q}$. 
Choose $Z=\C^r$, let $Q\in \C^{r\times r}$ be invertible and choose $G_1\in \C^{r\times r}$, $G_2\in \Lin(Y,\C^r)$, $K\in \Lin(\C^r,U)$ such that
\eq{
G_1=I, \quad G_2 = -(\Pop K_0Q)^\ast, \quad K = \eps K_0Q, \quad \mbox{and} \quad K_0z=\sum_{j=1}^r z_j\FFfun^{k_j}
}
for all $z=(z_1,\ldots,z_r)^T\in \C^r$. Then there exists $\eps^\ast>0$ such that for every $0<\eps\leq \eps^\ast$ the $r$-dimensional controller~\eqref{eq:fbcontrintro} solves the error feedback output regulation problem.
\end{thm}

Denote $\ydk[k](\cdot) = \Pdop\wdist^k$ for $k\in \List{q}$.
If $\yref(\cdot)\not\equiv 0$, we can choose an index set $ \set{0,k_2,\ldots,k_r} $ and 
\eq{
\Pop K_0 z
= \sum_{k=1}^r z_j \Pop\FFfun^{k_j} 
= z_1\yref + \sum_{j=2}^r z_j \ydk[k_j]
}
for all $z=(z_1,\ldots,z_r)^T\in \C^r$. Then the operator $G_2\in \Lin(Y,\C^r)$ is such that 
\eq{
G_2y
= -Q^\ast \pmat{\iprod{y}{\yref}_Y\\ \iprod{y}{\ydk[k_2]}_Y\\\vdots \\\iprod{y}{\ydk[k_r]}_Y}
\qquad
\qquad \forall
y\in \Lp[2](0,\persym;Y_0)
}
where $\iprod{f}{g}_Y=\int_0^\persym \iprod{f(t) }{g(t)}_{Y_0}dt$ for $f,g\in Y=\Lp[2](0,\persym;Y_0)$. On the other hand, if $\yref(\cdot)\equiv 0$, then we can choose an index set $ \set{k_1,\ldots,k_r}$ with $k_1\geq 1$ and then 
$G_2 y = -Q^\ast (\iprod{y}{\ydk[k_1]}_Y,\ldots,\iprod{y}{\ydk[k_r]}_Y)^T $ 
for all $y\in \Lp[2](0,\persym;Y_0)$.

The invertible matrix $Q\in \C^{r\times r}$ can be used to improve the stability of the closed-loop system. One possible choice of $Q$ is illustrated in the example in Section~\ref{sec:HarmOsc}. By Remark~\ref{rem:ORPsimpleALTG2} it would also be possible to choose any $G_2\in \Lin(Y,\C^r)$ such that $\gs(G_2\Pop K_0) \subset \C_-$. However, for our controller the choice $G_2 = -(\Pop K_0 Q)^\ast$ has a particularly simple structure.

We will see later that the exponent in the rate of decay of the regulation error is determined by the stability margin of the closed-loop system consisting of the lifted periodic system and the discrete-time controller. If available, this information can be used to choose a suitable value of the parameter $\eps>0$. As explained in Remark~\ref{rem:Oscepschoice}, for a finite-dimensional $X$ the closed-loop system operator and its spectrum can be approximated numerically by simulating the controlled system and recording the values of $x(\persym)$ and $z_1$ for when the initial state vectors $(x_0,z_0)^T$ are the Euclidean basis vectors of $X\times Z$. If $X$ is infinite-dimensional, the same procedure can be used for finite-dimensional approximations of the original system.

\subsection{Robust Output Regulation}
\label{sec:RORPcontroller}

In this section we present controllers for output tracking and disturbance rejection with the additional robustness requirement. The first controller presented in Theorem~\ref{thm:RORPcontroller} is a discrete-time feedback controller on an infinite-dimensional state-space. In fact, the internal model principle in Theorem~\ref{thm:IMP} will imply that robustness can not be achieved with a finite-dimensional autonomous feedback controller. However, Theorem~\ref{thm:RORPapprox} demonstrates that a finite-dimensional discrete-time feedback controller can be designed to achieve \keyterm{approximate} output tracking and disturbance rejection and robustness with respect to perturbations in the system.

\begin{thm}
  \label{thm:RORPcontroller}
  Assume the system is exponentially stable
and $\Pop\in \Lin(U,Y)$ is surjective.
Choose
\eq{
Z=Y = \Lp[2](0,\persym;Y_0), \qquad G_1 = I \in \Lin(Z),
}
let $G_2\in \Lin(Y,Z)$ be boundedly invertible
and let $K = \eps K_0$ where $K_0\in \Lin(Y,U)$ is such that $\gs(G_2\Pop K_0)\subset \C_-$.
Then there exists $\eps^\ast>0$ such that for every $0<\eps\leq \eps^\ast$ the infinite-dimensional controller~\eqref{eq:fbcontrintro}
solves the robust output regulation problem.
\end{thm}

Later in Section~\ref{sec:proofsmain} we will see that the assumption on the surjectivity of $\Pop$ is necessary for robustness due to the requirement that the transfer function of the lifted system must be surjective at the frequency $\mu =1$. For time-invariant systems this is a well-known condition, but for periodic systems it becomes fairly restrictive and can mainly be achieved in the situation where $D(t)$ are boundedly intertible for all $t\in [0,\persym]$. If the surjectivity assumption is satisfied, one possible choice for the stabilizing operator $K_0$ is $K_0=-(G_2\Pop)\pinv$, where $(G_2\Pop)\pinv$ is the Moore--Penrose pseudoinverse of $G_2\Pop$.

The following theorem introduces a simple finite-dimensional controller that solves the robust output regulation problem approximately in the sense that $\norm{y(n\persym+\cdot)-\yref(\cdot)}_{\Lp[2]}$ becomes small as $n\to \infty$. The asymptotic error bound presented in the theorem depends on the system, the reference and disturbance signals, as well as on the space $Y_N$ used in the construction.

\begin{thm}
\label{thm:RORPapprox}
Assume the periodic system is exponentially stable and let $Y_N$ be a finite-dimensional subspace of $Y$ such that $Y_N\subset \ran(\Pop)$. Choose $Z=\C^r$ with $r=\dim Y_N$,
\eq{
G_1=I\in \C^{r\times r}, \quad  K = \eps K_0, \quad \mbox{and} \quad G_2=G_{20}\YP\in \Lin(Y,Z)
}
  where $\YP$ is a projection onto $Y_N$, $G_{20}\in \Lin(Y_N,Z)$ is boundedly invertible, and $K_0\in \Lin(\C^r,U)$ is such that $\gs(G_2 \Pop K_0 )\subset \C_-$. Then there exists $\eps^\ast >0$ such that for all $0<\eps\leq \eps^\ast$ the controller~\eqref{eq:fbcontrintro} solves the output regulation problem approximately for all $\persym$-periodic reference and disturbance signals $\yref\in Y$ and $\wdist\in U_d$. Asymptotically the regulation error on the interval $[n\persym,(n+1)\persym)$ satisfies an estimate
\eq{
\norm{y(n\persym+\cdot)-\yref(\cdot)}_Y
\approx \norm{(I-\YP)\hspace{-.3ex}\left[ \Pop Kz + \Pdop\wdist-\yref \right] } 
}
where $z\in Z$ is the unique solution of $\YP\Pop Kz = \YP\yref - \YP\Pdop\wdist$.

Moreover, the controller is robust with respect to perturbations for which the perturbed closed-loop system is exponentially stable. If $1\in \rho(U_{\tilde{A}}(\persym,0))$ for the perturbed system, then the asymptotic error is of the form above for perturbed parameters of the system.
\end{thm}

To estimate the asymptotic regulation error, the norms $\norm{(I-\YP)\Pop Kz}$ and $\norm{(I-\YP)\Pdop\wdist}$ can be approximated based on measurements from the system using the procedure in Section~\ref{sec:P1measurement}, and $\norm{(I-\YP)\yref}$ can be computed explicitly.

\subsection{Connection to Periodic Regulator Equations}
\label{sec:PerRegEqns}

In this section we relate the solution of the output regulation problem in Theorem~\ref{thm:ORPstatic} to the solvability of the \keyterm{periodic regulator equations} studied in~\cite{PauPoh12b}. In particular, it was shown in~\cite[Thm. 5]{PauPoh12b} that the control law solving the output regulation problem for an autonomous system 
and for reference signals $\yref(\cdot)$ and disturbance signals $\wdist(\cdot)$ can be expressed in terms of the unique periodic mild solution $(\Pi(\cdot),\Gamma(\cdot))$ of the equations
\begin{subequations}
  \label{eq:PerRE}
  \eqn{
  \label{eq:PerRE1}
  \dot{\Pi}(t) + \Pi(t)S(t)&=A(t)\Pi(t) + B(t)\Gamma(t) + B_d(t)E(t)\\
  \label{eq:PerRE2}
  0&= C(t)\Pi(t) + D(t)\Gamma(t) + F(t)
  }
\end{subequations}
where the periodic exosystem generating the signals $\yref(\cdot)$ and $\wdist(\cdot)$ is of the form
\begin{subequations}
  \label{eq:perexo}
  \eqn{
  \label{eq:perexo1}
  \dot{v}(t)&=S(t)v(t), \qquad v(0)=v_0\in\C^q\\
  \wdist(t)&=E(t)v(t)\\
  \yref(t)&=-F(t)v(t).
  }
\end{subequations}
In the situation of Theorem~\ref{thm:ORPstatic} the signals $\yref(\cdot)$ and $\wdist(\cdot)$ can be generated with choices $q=1$, $S(\cdot)\equiv 0$, $E(\cdot) = \wdist$ and $F(\cdot)=-\yref$. By ``periodic mild solution'' of~\eqref{eq:PerRE} we mean that $\Pi(\cdot)$ and $\Gamma(\cdot)$ are periodic functions, 
\eq{
\Pi(t) =U_A(t,0)\Pi(0)U_S(0,t) + \int_0^t U_A(t,s)(B(s)\Gamma(s)+B_d(s)E(s))U_S(s,t)ds,
}
and~\eqref{eq:PerRE2} is satisfied on $[0,\persym]$. Here $U_S(t,s)$ is the fundamental matrix of the differential equation~\eqref{eq:perexo1}. 
It was shown in~\cite[Thm. 5]{PauPoh12b} that output regulation for a stable autonomous system is achieved with a control input $u(t)=\Gamma(t)v(t)$.
The following theorem shows that the control law presented in Theorem~\ref{thm:ORPstatic} for the nonautonomous system~\eqref{eq:plantintro} is of the same form.

\begin{thm}
  \label{thm:PerREconnection}
  The control law $u(\cdot)$ in Theorem~\textup{\ref{thm:ORPstatic}} is of the form $u(\cdot)=\Gamma(\cdot)v(\cdot)$ where the pair $(\Pi(\cdot),\Gamma(\cdot))$ is a periodic mild solution of the equations~\eqref{eq:PerRE} and $v(t)$ is the state of the exosystem~\eqref{eq:perexo} generating the reference and disturbance signals.
\end{thm}

\begin{proof}
The signals $\yref(\cdot)$ and $\wdist(\cdot)$ can be generated by choosing $q=1$, $S(t)\equiv 0$, $E(\cdot) = \wdist(\cdot)$, $F(\cdot)=-\yref(\cdot)$ and $v_0=1$. Denote by $\Gamma(\cdot)=u(\cdot) \in \Lploc[2](0,\infty;U_0)$ the $\persym$-periodic function satisfying $u(\cdot)=\FFfun(\cdot)$ on $[0,\persym]$ where 
$\Pop\FFfun=\yref-\Pdop\wdist$. Denote $f(\cdot) = B(\cdot)\Gamma(\cdot) + B_d(\cdot)\wdist(\cdot)$ and define $\Pi(\cdot)\in C(\R,\Lin(\C,X))=C(\R,X)$ such that for all $t\in[0,\persym]$
\eq{
\Pi(t) = U_A(t,0)\int_0^\persym (1-U_A(\persym,0))\inv U_A(\persym,s)f(s)ds + \int_0^t U_A(t,s)f(s)ds.
}
Since $S(\cdot)\equiv 0\in\C$, we have $U_S(t,s)=1\in\C$ for all $t\geq s$, and the form of $\Pi(\cdot)$ implies that it is the mild solution of~\eqref{eq:PerRE1} on $[0,\persym]$. Moreover, since $f(\cdot)$ is $\persym$-periodic and $U_A(t+\persym,s+\persym)=U_A(t,s)$ for all $t\geq s$,
a direct computation shows that also $\Pi(\cdot)$ is $\persym$-periodic. Thus $\Pi(\cdot)$ is a periodic mild solution of~\eqref{eq:PerRE1}. Finally, the definitions of the operator-valued functions $\Pi(\cdot)$ and $\Gamma(\cdot)$ imply that on $[0,\persym]$ we have
\eq{
 C(\cdot)\Pi(\cdot)+D(\cdot)\Gamma(\cdot) + F(\cdot)
= \Pop u + \Pdop \wdist - \yref = 0.
}
\end{proof}

Similarly, the functions $\FFfun^k$ in Theorem~\ref{thm:ORPfeedback} are related to solutions of periodic regulator equations of the form~\eqref{eq:PerRE}.

\begin{cor}
  \label{cor:PerREconnection}
  The periodic extensions $u^k(\cdot)$ of the functions $\FFfun^k$ in Theorem~\textup{\ref{thm:ORPfeedback}} are of the form $u^k(\cdot)=\Gamma_k(\cdot)$ where $\Gamma_k(\cdot)$ satisfies the following.
  \begin{itemize}
    \item[\textup{(a)}] If $k=0$, then  $(\Pi_k(\cdot),\Gamma_k(\cdot))$ is a periodic mild solution of~\eqref{eq:PerRE} with $q=1$, $S(\cdot)\equiv 0$, $E(\cdot)\equiv 0$ and $F(\cdot)=-\yref$.
    \item[\textup{(b)}] If $k\neq 0$, then $(\Pi_k(\cdot),\Gamma_k(\cdot))$ is a periodic mild solution of~\eqref{eq:PerRE} with $q=1$, $S(\cdot)\equiv 0$, $E(\cdot)= \wdist^k$ and $F(\cdot)\equiv 0$.
  \end{itemize}
\end{cor}

\section{The Proofs of The Main Results}
\label{sec:proofsmain}

In this section we present the proofs of the theorems in Section~\ref{sec:contrdesign}. We begin by representing the periodic system~\eqref{eq:plantintro} as a discrete-time system using the lifting technique~\cite{BamPea91}. In Section~\ref{sec:discreteRORP} we recall and extend the theory of output regulation for infinite-dimensional discrete-time systems and in particular introduce the internal model principle for systems with infinite-dimensional input and output spaces. The proofs of the theorems in Section~\ref{sec:contrdesign} are based on combining the results in Sections~\ref{sec:liftedsys} and~\ref{sec:discreteRORP} and they are presented in Section~\ref{sec:proofs}.

\subsection{The Lifted System}
\label{sec:liftedsys}

The ``lifted system''~\cite{BamPea91} corresponding to the periodic system~\eqref{eq:plantintro} is an autonomous discrete-time system
\begin{subequations}
  \label{eq:plantlifted}
  \eqn{
  \Lted{x}_{n+1}&= \Ltedop{A}\Lted{x}_n + \Ltedop{B}\Lted{u}_n+ \Ltedop{B}_d\Lted{w}_n, \qquad 
  \Lted{x}_0 = x_0\in X\\
  \Lted{y}_n &= \Ltedop{C}\Lted{x}_n + \Ltedop{D}\Lted{u}_n+ \Ltedop{D}_d\Lted{w}_n,
  }
\end{subequations}
on the space $X$, where the 
lifted state~$\Lted{x}_n$, the inputs~$\Lted{u}_n$ and $\Lted{w}_n$ and the output~$\Lted{y}_n$ are given by
\eq{
\Lted{x}_n &= x(n\persym), \\
\Lted{u}_n &= u(n\persym+\cdot)\in U = \Lp[2](0,\persym;U_0)\\
\Lted{w}_n &= w(n\persym+\cdot)\in U_d = \Lp[2](0,\persym;U_{d0})\\
\Lted{y}_n &= y(n\persym+\cdot)\in Y = \Lp[2](0,\persym;Y_0).
}
The operators 
$\Ltedop{A}\in \Lin(X)$, $\Ltedop{B}\in \Lin(U,X)$, $\Ltedop{B}_d\in \Lin(U_d,X)$, $\Ltedop{C}\in \Lin(X,Y)$, $\Ltedop{D}\in \Lin(U,Y)$, and $\Ltedop{D}_d\in \Lin(U_d,Y)$ are such that for all $\Lted{x}\in X$, $\Lted{u}\in U$, and $\Lted{w}\in U_d$ 
\eq{
\Ltedop{A}\Lted{x} &= \FM_A(\persym,0)\Lted{x}\\
\Ltedop{B}\Lted{u} &= \int_0^\persym\FM_A(\persym,s)B(s)\Lted{u}(s)ds\\
\Ltedop{B}_d\Lted{w} &= \int_0^\persym\FM_A(\persym,s)B_d(s)\Lted{w}(s)ds\\
\Ltedop{C}\Lted{x} 
&= C(\cdot)\FM_A(\cdot,0)\Lted{x}\\
\Ltedop{D}\Lted{u} &=  C(\cdot)\int_0^\cdot \FM_A(\cdot,s)B(s)\Lted{u}(s)ds +D(\cdot)\Lted{u}(\cdot)  \\
\Ltedop{D}_d\Lted{w} &=  C(\cdot)\int_0^\cdot \FM_A(\cdot ,s)B_d(s)\Lted{w}(s)ds .
}
The lifted system~\eqref{eq:plantlifted} is an autonomous discrete-time system on the Banach space $X$ with infinite-dimensional input and output spaces. 
Due to Lemma~\ref{lem:UAexpstabchar} and
$\Ltedop{A}=\FM_A(\persym,0)$ the lifted system is exponentially stable if and only if~\eqref{eq:plantintro} is exponentially stable.
We denote the transfer functions of the lifted system with
\eq{
\Ltedop{P}(\mu) &= \Ltedop{C}(\mu I-\Ltedop{A})\inv \Ltedop{B} + \Ltedop{D} \in \Lin( U, Y), 
\\
\Ltedop{P}_d(\mu) &= \Ltedop{C}(\mu I-\Ltedop{A})\inv \Ltedop{B}_d + \Ltedop{D}_d \in \Lin( U, Y)
}
for all $\mu\in\rho(\Ltedop{A})$. It is immediate that the operators $\Pop$ and $\Pdop$ are related to the lifted system by $\Pop = \Ltedop{P}(1)$ and $\Pdop=\Ltedop{P}_d(1)$.
Finally, for a $\persym$-periodic $\yref\in \Lploc[2](0,\infty;Y_0)$ we have
\eq{
\int_{n\persym}^{(n+1)\persym} \hspace{-2ex} \norm{y(t)-\yref(t)}^2 dt 
=\norm{\Lted{y}_n - \yref(\cdot)}_Y^2  \qquad 
n\geq 0.
}

\subsection{Controllers for Discrete-Time Systems}
\label{sec:discreteRORP}

In this section we recall and extend selected results for output regulation of a stable discrete-time system 
\begin{subequations}
  \label{eq:discrplant}
  \eqn{
  x_{n+1} &= Ax_n+Bu_n+B_dw_n, \qquad x_0\in X\\
  y_n &= Cx_n+Du_n + D_dw_n
  }
\end{subequations}
on a Banach space $X$. Here $A\in \Lin(X)$, $B\in \Lin(U,X)$, $B_d\in \Lin(U_d,X)$, $C\in \Lin(X,Y)$ and $D\in \Lin(U,Y)$. 
In this subsection $U$, $U_d$, and $Y$ may be general Hilbert spaces. We denote
\eq{
P(\mu) = CR(\mu,A)B+D , \qquad P_d(\mu)=CR(\mu,A)B_d + D_d
}
for $\mu\in\rho(A)$. We consider output tracking and disturbance rejection of signals generated by a discrete-time exosystem 
\begin{subequations}
  \label{eq:discrexo}
  \eq{
  v_{n+1}&= Sv_n, \qquad v_0\in \C^q\\
  w_n&=Ev_n\\
  \yrefn&=-Fv_n
  }
\end{subequations}
with $S=\diag(\mu_1,\ldots,\mu_q)\in \C^{q\times q}$, $\set{\mu_k}_{k=1}^q\subset \C$ with $\abs{\mu_k}=1$ for all $k\in\List{q}$, $E\in \Lin(\C^q,U_d)$, and $F\in \Lin(\C^q,Y)$.

We consider an error feedback controller of the form
\begin{subequations}
  \label{eq:contrdiscr}
  \eqn{
z_{n+1} &= G_1 z_n + G_2 e_n, \qquad z_0\in Z, \\
  u_n&= Kz_n
  }
\end{subequations}
on a Banach space $Z$. Here $G_1 \in \Lin(Z)$, $G_2\in \Lin(Y,Z)$, $K\in \Lin(Z,U)$, and $e_n = y_n-\yrefn$. 
The closed-loop system with the state $x_e^n=(x_n,z_n)^T\in X\times Z$ is of the form
\begin{subequations}
  \label{eq:CLsysdiscr}
  \eqn{
  x_{n+1}^e &= A_ex_n^e + B_e v_n, \qquad x_0^e=(x_0,z_0)^T\\
  e_n &= C_ex_n^e + D_e v_n
  }
\end{subequations}
with
  \eq{
  A_e &= \pmat{A& BK\\G_2C&G_1+G_2 DK}, \quad
   B_e =\pmat{B_dE\\G_2 (D_dE+F)},
  }
  $C_e=[C, ~ DK]$, and $D_e=D_dE+F$.  
The closed-loop system is exponentially stable if and only if $\abs{\gl}<1$ for all $\gl\in\gs(A_e)$.

\begin{dfn}
  In the \keyterm{output regulation problem} for the system~\eqref{eq:discrplant} and the exosystem~\eqref{eq:discrexo} the goal is to choose the controller~\eqref{eq:contrdiscr} in such a way that the closed-loop system~\eqref{eq:CLsysdiscr} is exponentially stable and there exist $M,\ga >0$ such that
  \eqn{
  \label{eq:discrerrconv}
  \norm{y_n-\yrefn} \leq Me^{-\ga n}(\norm{x_0}+\norm{z_0}+\norm{v_0})
  }
  for all initial states $x_0\in X$, $z_0\in Z$, and $v_0\in \C^q$.

In \keyterm{robust output regulation} we in addition require that if the parameters $\sysops$ are perturbed to $\sysopspert$ in such a way that the exponential closed-loop stability is preserved, then~\eqref{eq:discrerrconv} holds for some constants $M,\ga>0$ and 
for all initial states.
\end{dfn}

\begin{thm}
  \label{thm:ORPchar}
  Assume that $\set{\mu_k}_{k=1}^q\subset \rho(A)$ and that the closed-loop system with the controller $(G_1,G_2,K)$
is exponentially stable.
Then $(G_1,G_2,K)$ solves the output regulation problem if and only if for all $k\in \List{q}$ the equations
\begin{subequations}
  \label{eq:ORPchareqns}
  \eqn{
  P(\mu_k)Kz_k &= -F\phi_k - P_d(\mu_k)E\phi_k\\
  (\mu_k-G_1)z_k &= 0
  }
\end{subequations}
have solutions $z_k\in Z$. Here $\set{\phi_k}_{k=1}^q$ denotes the Euclidean basis of $\C^q$.
\end{thm}

\begin{proof}
  Exactly as in the continuous-time case in~\cite{HamPoh10,PauPoh10} a controller stabilizing the closed-loop system solves the 
output regulation problem if and only if the ``regulator equations''
  \begin{subequations}
    \label{eq:regeqnsdiscr}
    \eqn{
    \label{eq:regeqnsdiscr1}
    \Sigma S &= A_e \Sigma + B_e\\
    \label{eq:regeqnsdiscr2}
    0& = C_e\Sigma + D_e.
    }
  \end{subequations}
  have a solution $\Sigma\in \Lin(\C^q,X\times Z)$.
The operators $(A_e,B_e,C_e,D_e)$
of the closed-loop system~\eqref{eq:CLsysdiscr} 
and the regulator equations~\eqref{eq:regeqnsdiscr} are of the same form as in the continuous-time case. Since $\set{\mu_k}_{k=1}^q\subset \rho(A)$,
the equivalence of the solvability of~\eqref{eq:ORPchareqns} and the solvability of~\eqref{eq:regeqnsdiscr} can be shown as in the proofs of~\cite[Thm. 4]{PauPoh13a} and~\cite[Thm. 5.1]{PauPoh14a}.  
\end{proof}

Theorem~\ref{thm:ORPchar} implies that for output regulation it is necessary that 
$F\phi_k + P_d(\mu_k)E\phi_k\in \ran(P(\mu_k))$ for all $k$.

\begin{thm}
  \label{thm:ORPsimplecontr}
  Assume the system~\eqref{eq:discrplant} is exponentially stable and $S=I\in\C^{q\times q}$. Assume further that $F\phi_k + P_d(1)E\phi_k\in \ran(P(1))$ for all $k\in \List{q}$ and let $\set{u_k}_{k=1}^r\subset U$ be a minimal set of linearly independent vectors such that 
\eq{
F\phi_k + P_d(1)E\phi_k \in \Span \set{P(1)u_j}_{j=1}^r, \quad   \forall k\in \List{q}.
}
Choose the controller~\eqref{eq:contrdiscr} on $Z=\C^r$ in such a way that $G_1=I\in\C^{r\times r}$, $G_2=-(P(1)K_0Q)^\ast\in \Lin(Y,\C^r)$, and $K=\eps K_0Q$ where $K_0 = [u_1,\ldots,u_r]\in \Lin(\C^r,U)$, $Q\in \C^{r\times r}$ is invertible, and
$\eps>0$.
Then there exists $\eps^\ast>0$ such that for all $0<\eps\leq \eps^\ast$ the controller solves the output regulation problem.
\end{thm}

\begin{proof}
  The structures of $G_1$ and $K_0$ imply that the equations~\eqref{eq:ORPchareqns} have solutions for all $k\in\List{q}$. 

  Similarly as in the proof of~\cite[Thm. 8]{Pau16a} we can show that the closed-loop system operator is boundedly similar to the operator
  \eq{
\pmat{A+\eps H G_2C&0\\-G_2C & \hspace{-2ex} 1+\eps G_2P(1)K_0Q} - \eps^2 \pmat{0&\hspace{-.5ex}HG_2P(1)K_0Q\\0&0}
}
where $H=-R(1,A)BK_0Q\in \Lin(Z,X)$
and $G_2P(1)K_0Q = -G_2 G_2^\ast $.
Since $\set{u_j}_{j=1}^r$ was chosen in such a way that $\set{P(1)u_j}_{j=1}^r$ is linearly independent, we have that $G_2$ is surjective and $G_2 G_2^\ast>0$. This implies that for small $\eps>0 $ the spectral radius of $1-\eps G_2 G_2^\ast$ is smaller than~$1$ and $\eps \norm{R(\mu,1-\eps G_2 G_2^\ast)}\leq M_0$ for some $M_0>0$. 
The above block operator is of the form $A_0(\eps)+\Delta(\eps)$, and
the fact that we can choose $\eps^\ast>0$ so that the closed-loop system is exponentially stable whenever $0<\eps\leq \eps^\ast$ follows from studying $\Delta(\eps)R(\mu,A_0(\eps))$ for $\mu\in\C$ with $\abs{\mu}\geq 1$.
\end{proof}

\begin{rem}
  \label{rem:ORPsimpleALTG2}
  Instead of $G_2=-(P(1)K_0Q)^\ast$, we could take any $G_2\in \Lin(Y,\C^r)$ such that the eigenvalues of the matrix $G_2P(1)K_0Q \in \C^{r\times r}$ have negative real parts, analogously as in~\cite[App. B]{HamPoh11}.
\end{rem}

The last two results in this section concern the robust output regulation problem. The following theorem presents the internal model principle for infinite-dimensional discrete-time systems. 
Conditions~\eqref{eq:Gconds} and~\eqref{eq:discrpcopy} offer two alternative definitions for an ``internal model''. Condition~\eqref{eq:discrpcopy} is a direct generalization of the classical internal model of Francis and Wonham~\cite{FraWon75a} and Davison~\cite{Dav76}, whereas the conditions~\eqref{eq:Gconds} studied in~\cite{HamPoh10,PauPoh10} have the advantage of being applicable for systems with infinite-dimensional output spaces.

\begin{thm}
  \label{thm:IMP}
  Assume the closed-loop system with the controller $(G_1,G_2,K)$ is exponentially stable. Then $(G_1,G_2,K)$ solves the robust output regulation problem if and only if 
  \begin{subequations}
    \label{eq:Gconds}
    \eqn{
    \ran(\mu_k-G_1)\cap \ran(G_2)&=\set{0}, \qquad \forall k\in \List{q}\\
    \ker(G_2)&=\set{0}.
    }
  \end{subequations}
 In particular, a stabilizing controller can solve the robust output regulation problem only if 
  \eqn{
  \label{eq:discrpcopy}
  \dim\ker(\mu_k-G_1)\geq \dim Y, 
  \qquad \quad \forall k\in \List{q}.
  }
  Finally, if $\dim Y<\infty$, then the condition~\eqref{eq:discrpcopy} is also sufficient for the robustness of the controller.
\end{thm}

\begin{proof}
Since the operators $(A_e,B_e,C_e,D_e)$ of the closed-loop system~\eqref{eq:CLsysdiscr} and the regulator equations~\eqref{eq:regeqnsdiscr} are of the same form as in the continuous-time case, 
the proof can be completed as in~\cite[Sec. 4--6]{PauPoh10} (see also~\cite[Thm. 7]{Pau16a}).
\end{proof}

The following controller is a discrete-time special case of the one presented in~\cite{HamPoh11}, and the structure is also related to the controllers in~\cite{LogTow97,HamPoh00,UkaIwa90} where $\dim Y<\infty$.

\begin{thm}
  \label{thm:RORPsimplecontr}
  Assume the system is exponentially stable, the exosystem is such that $q=1$ and $S=1\in\C$, and $P(1)$ is surjective.
  Choose a Hilbert space $Z$ and the parameters $(G_1,G_2,K)$  in such a way that $G_1=I_Z\in \Lin(Z)$, $G_2\in \Lin(Y,Z)$ is boundedly invertible, and $K=\eps K_0$ where $\eps>0$ and $K_0\in \Lin(Y,U)$ is such that  $\gs(G_2 P(1)K_0)\subset \C_-$. Then there exists $\eps^\ast>0$ such that for all $0<\eps\leq \eps^\ast$ the controller solves the robust output regulation problem.
\end{thm}

\begin{proof}
The conditions~\eqref{eq:Gconds} are satisfied since $G_1 = I$ and $G_2$ is boundedly invertible. We can again show that the closed-loop system operator $A_e$ is similar to the block operator in the proof of Theorem~\ref{thm:ORPsimplecontr} with $Q=I$. Since $\gs(G_2 P(1)K_0)\subset \C_-$, the stability of the closed-loop system for small $\eps>0$ can be shown as in the proof of Theorem~\ref{thm:ORPsimplecontr}.
\end{proof}

\subsection{The Proofs of the Main Theorems}
\label{sec:proofs}

We can now combine the results in Sections~\ref{sec:liftedsys} and~\ref{sec:discreteRORP} to present the proofs of Theorems~\ref{thm:ORPstatic},~\ref{thm:ORPfeedback}, and~\ref{thm:RORPcontroller}.

\begin{proof}[Proof of Theorem~\textup{\ref{thm:ORPstatic}}]
Consider the lifted version~\eqref{eq:plantlifted} of the periodic system~\eqref{eq:plantintro} in the situation
   where $(\Lted{w}_n)_{n\geq 0}\subset U_d=\Lp[2](0,\persym;U_{d0})$ is such that $\Lted{w}_n=\wdist(\cdot)$ for all $n\geq 0$.
Choose $(\Lted{u}_n)_{n\geq 0}\subset U=\Lp[2](0,\persym;U_0)$ such that $\Lted{u}_n=\FFfun(\cdot)$ for all $n\geq 0$,
where $\FFfun\in U$ is such that $\Pop \FFfun = \yref - \Pdop \wdist$. Since the lifted system~\eqref{eq:plantlifted} is stable and since $(\Lted{u}_n)_{n\geq 0}$ and $(\Lted{w}_n)_{n\geq 0}$ are constant signals, it is well-known that the output of~\eqref{eq:plantlifted} satisfies
  \eq{
  \Lted{y}_n \stackrel{n\to\infty}{\longrightarrow} \Ltedop{P}(1)\FFfun + \Ltedop{P}_d(1)\wdist =
 \Pop\FFfun + \Pdop \wdist 
 =\yref(\cdot).
  }
  Moreover, since the lifted system is exponentially stable,  $\norm{\Lted{y}_n-\yref(\cdot)}_Y\leq Me^{-\ga n}\sqrt{\norm{x_0}^2+1}$ for some $M,\ga>0$.

  If the periodic system~\eqref{eq:plantintro} is only strongly stable and $1\in\rho(U_A(\persym,0))$, we have $1\in \rho(\Ltedop{A})$ and~\eqref{eq:plantlifted} is strongly stable in the sense that $\Ltedop{A}^n \Lted{x}\to 0$ as $n\to \infty$ for all $\Lted{x}\in X$. In this situation $\norm{\Lted{y}_n-\yref(\cdot)}_Y\to 0$ as $n\to \infty$ for all $x_0\in X$.

Finally, the converse statement follows from the property that if the input~$u$ is the $\persym$-periodic extension of $u_0\in\Lp[2](0,\persym;U_0)$, then the output satisfies $\norm{\Lted{y}_n-(\Pop u_0 + \Pdop\wdist)}_Y\to 0$ as $n\to \infty$.
\end{proof}

To prove Theorems~\ref{thm:ORPfeedback}--\ref{thm:RORPapprox} we need to show that the closed-loop stability in the sense of Section~\ref{sec:controlobjectives} is equivalent to the exponential stability of the discrete-time closed-loop system.

\begin{lem}
  \label{lem:CLstab}
  The closed-loop system consisting of the lifted system~\eqref{eq:plantlifted} and a discrete-time controller~\eqref{eq:contrdiscr} is exponentially stable if and only if there exist $M_0,M_1,\ga_0,\ga_1>0$ such that in the case where $\yref(\cdot)\equiv 0$ and $\wdist(\cdot)\equiv 0$ we have
  \begin{subequations}
    \label{eq:discrCLdecay}
  \eqn{
    \label{eq:discrCLdecay1}
  \norm{x(t)}&\leq M_0e^{-\ga_0 t}(\norm{x_0}+\norm{z_0}),\\
    \label{eq:discrCLdecay2}
  \qquad \norm{z_n}&\leq M_1e^{-\ga_1 n}(\norm{x_0}+\norm{z_0})
  }
  \end{subequations}
  for all $t\geq 0$ and $n\geq 0$ and for all
  $x_0\in X$ and $z_0\in Z$.
\end{lem}

\begin{proof}
  The state of the closed-loop system satisfies
    \eq{
\pmat{\Lted{x}_{n+1}\\z_{n+1}} \hspace{-.5ex}=\hspace{-.5ex} \pmat{\Ltedop{A}&\Ltedop{B}K\\ G_2 \Ltedop{C}&\hspace{-1ex}G_1 + G_2 \Ltedop{D}K}\pmat{\Lted{x}_n\\ z_n} + \pmat{\Ltedop{B}_d\wdist\\ G_2(\Ltedop{D}_d\wdist-\yref)}
    }
    with initial state $(\Lted{x}_0,z_0)^T\in X\times Z$.
  The ``if'' part follows directly from the fact that $\Lted{x}_n= x(n\persym)$ for all $n\geq 0$.
  On the other hand, if the discrete closed-loop system is stable, there exist $M_1,M_2,\ga_1,\ga_2>0$ such that $\norm{z_n} \leq M_1 e^{-\ga_1 n}(\norm{x_0}+\norm{z_0})$ and $\norm{x(n \persym)}\leq M_2e^{-\ga_2 n}(\norm{x_0}+\norm{z_0})$ for all $n\geq 0$, and thus~\eqref{eq:discrCLdecay2} holds. 
  If $t=n\persym + t_0$ for some $n\geq 0$ and $0\leq t_0<\persym $, then 
  the periodicity and exponential stability of~\eqref{eq:plantintro} 
  together with $u(n\persym+\cdot)=\Lted{u}_n=Kz_n$
  imply that 
  \eq{
  \norm{x(t)}
  &=\Norm{U_A(t_0,0)x(n\persym) + \int_0^{t_0}\hspace{-.7ex} U_A(t_0,s)B(s)(Kz_n)(s)ds}\\
  &\leq (M_3 e^{-\ga_2 n} + M_4 \norm{K}e^{-\ga_1 n})(\norm{x_0}+ \norm{z_0})
  }
  for some constants $M_3,M_4>0$ independent of $t_0\in[0,\persym)$ and $n\geq 0$. 
  From this it follows that there exists $M_0>0$ such that also~\eqref{eq:discrCLdecay1} holds with $\ga_0=\min\set{\ga_1,\ga_2}/\persym$.
\end{proof}

\begin{proof}[Proof of Theorem~\textup{\ref{thm:ORPfeedback}}]
  The controller~\eqref{eq:fbcontrintro} is of the form~\eqref{eq:contrdiscr} where $e_n = y(n\persym+\cdot)-\yref(\cdot)=\Lted{y}_n-\yref$ on $[0,\persym]$. The $\persym$-periodic reference and disturbance signals can be expressed as constant discrete-time signals $\yrefnL\equiv \yref(\cdot)\in Y$ and $\wdistnL\equiv \sum_{k=1}^q v_k\wdist^k(\cdot)\in U_d$ and they can be generated with a $q+1$-dimensional exosystem
\eq{
\Lted{v}_{n+1} &=\Lted{v}_n, \qquad \Lted{v}_0\in\C^{q+1}\\
\wdistnL &= E\Lted{v}_n\\
\yrefnL &=-F\Lted{v}_n
}
satisfying $F\phi_0 = -\yref(\cdot)\in Y$ and $E\phi_k=\wdist^k(\cdot)\in U_d$ where $\set{\phi_k}_{k=0}^q$ is the Euclidean basis of $\C^{q+1}$, and $\Lted{v}_0 = (1,\vdvec^T)^T$.
Since $\Pop = \Ltedop{P}(1)$ and $\Pdop = \Ltedop{P}_d(1)$, Theorem~\ref{thm:ORPsimplecontr} implies that the controller with the choices of $(G_1,G_2,K)$ in Theorem~\ref{thm:ORPfeedback} solves the output regulation problem for the lifted system~\eqref{eq:plantlifted}.
In particular, for all intial states $x_0\in X$, $z_0\in Z$, and for all $\vdvecdef\in \C^q$
we have 
\eq{
\MoveEqLeft[4] \int_{n\persym}^{(n+1)\persym}\norm{y(t)-\yref(t)}^2dt
= \norm{\Lted{y}_n - \yrefnL}^2 \\
&\leq M^2e^{-2\ga n}\left( \norm{x_0}^2 + \norm{z_0}^2 + \norm{\vdvec}^2
+1 \right)
}
for some constants $M,\ga >0$ and for all $n\geq 0$. Finally, by Lemma~\ref{lem:CLstab} the closed-loop system is exponentially stable in the appropriate sense.
\end{proof}

\begin{proof}[Proof of Theorem~\textup{\ref{thm:RORPcontroller}}]
Analogously as in the proof of Theorem~\ref{thm:ORPfeedback} the first two parts of the robust output regulation problem follow from a direct application of Theorem~\ref{thm:RORPsimplecontr}. Also the third part of the problem is satisfied since by Lemma~\ref{lem:CLstab} the stability of the closed-loop system is equivalent to the stability of the lifted system with the discrete-time controller. Thus 
$\norm{\Lted{y}_n-\yrefL}_Y\to 0$ at exponential rates as $n\to \infty$ for all perturbations preserving the closed-loop stability and for all  $x_0\in X$, $z_0\in Z$, and $\vdvec\in \C^q$.
\end{proof}

\begin{proof}[Proof of Theorem~\textup{\ref{thm:RORPapprox}}]
The $\persym$-periodic signals $\yref(\cdot)$ and $\wdist(\cdot)$ are generated by a $1$-dimensional discrete-time exosystem with $S=1$,  $E=\wdist$, $F = -\yref$, and initial state $v_0=1$.
The full closed-loop system consisting of the lifted system $(\Ltedop{A},\Ltedop{B},\Ltedop{B}_d,\Ltedop{C},\Ltedop{D})$ and the controller $(I,G_{20}\YP,\eps K_0)$ is of the form~\eqref{eq:CLsysdiscr} with
  $x_n^e=(\Lted{x}_n,z_n)^T\in X\times Z$, 
  \eq{
  A_e = \pmat{\Ltedop{A}& \Ltedop{B}K\\G_{20}\YP\Ltedop{C}&I+G_{20}\YP \Ltedop{D}K}, \quad
   B_e =\pmat{\Ltedop{B}_d\wdist\\G_{20}\YP (\Ltedop{D}_d\wdist-\yref)},
  }
  $C_e=[\Ltedop{C}, ~ \Ltedop{D}K]$, and $D_e=\Ltedop{D}_d\wdist-\yref$. 
  Since $\gs(G_2 \Ltedop{P}(1) K_0)\subset \C_-$, the exponential stability of the closed-loop system for all sufficiently small $\eps>0$ can be shown as in the proof of Theorem~\ref{thm:ORPsimplecontr}.
  
  If we denote $P_e(\mu) = C_e R(\mu,A_e)B_e + D_e$, then the closed-loop stability and $v_n\equiv 1$ imply that the regulation error satisfies $e_n\to P_e(1)$ and
  \eq{
  \norm{y(n\persym+\cdot)-\yref(\cdot)}_Y = \norm{e_n} \to \norm{P_e(1)}
  }
  as $n\to \infty$. The first part of the proof is complete once we show that $P_e(1)=(I-\YP)(\Pop Kz + \Pdop \wdist - \yref)$ where $z\in Z$ is such that
  \eqn{
  \label{eq:RORPapprzeq}
  \YP\Pop Kz = \YP \yref - \YP \Pdop\wdist.
  }
  Since $\YP\Pop K\in \Lin(\C^r,Y_N)$ is surjective and $r=\dim Y_N$, equation~\eqref{eq:RORPapprzeq} has a unique solution.
  To compute $P_e(1)$, denote $(x,z)^T=R(1,A_e)B_e$. We then have
  \eq{
\pmat{I-\Ltedop{A}&-\Ltedop{B}K\\-G_{20}\YP\Ltedop{C}&\hspace{-1ex}-G_{20}\YP\Ltedop{D}K}\hspace{-1ex}\pmat{x\\z}   \hspace{-.5ex}=\hspace{-.5ex} \pmat{\Ltedop{B}_d \wdist\\ G_{20}\YP(\Ltedop{D}_d\wdist-\yref) }.
  }
  Since $1\in \rho(\Ltedop{A})$ the first equation implies $x = R(1,\Ltedop{A})(\Ltedop{B}Kz + \Ltedop{B}_d \wdist)$. Substituting $x$ to the second equation and using the invertibility of $G_{20}$ shows that $z\in Z$ is the unique solution of~\eqref{eq:RORPapprzeq}. Finally, a direct computation using~\eqref{eq:RORPapprzeq} shows that
  \eq{
  P_e(1) & = C_e \pmatsmall{x\\z} + D_e
  = \Pop Kz + \Pdop \wdist - \yref\\
  &= (I-\YP)(\Pop Kz + \Pdop \wdist - \yref).
  }

  If the parameters of the periodic system are perturbed in such a way that the exponential closed-loop stability is preserved, then for any signals $\yref(\cdot)$ and $\wdist(\cdot)$ the regulation error satisfies $\norm{e_n}\to \norm{\tilde{P}_e(1)}$, where $\tilde{P}_e(\mu)$ is the transfer function of the perturbed closed-loop system. If we also have $1\in \rho(\tilde{\Ltedop{A}})$, then 
   we can show analogously as above that
  $\tilde{P}_e(1)=(I-\YP)(\Poppert Kz + \Pdoppert \wdist - \yref)$ where $z\in Z$ is such that 
  $\YP\Poppert Kz = \YP \yref - \YP \Pdoppert\wdist$.  
\end{proof}

\section{Measuring $\Pop$ and $\Pdop \wdist^k$ From The System}
\label{sec:P1measurement}

In this section we introduce a simple method for approximating the operator $\Pop \in \Lin(U,Y)$ and the functions $\Pdop \wdist^k\in Y$ based on measurements from the output of the original periodic system~\eqref{eq:plantintro}.  Throughout this section we assume $U=\Lp[2](0,\persym;U_0)$ has an orthonormal basis $\set{\varphi_k}_{k=1}^\infty$ and $Y=\Lp[2](0,\persym;Y_0)$ has an orthonormal basis $\set{\psi_k}_{k=1}^\infty$.

It is well-known that since the discrete-time lifted system~\eqref{eq:plantlifted} is stable, the output $\Lted{y}_n$ corresponding to any initial state $x_0\in X$, the constant input $\Lted{u}_n\equiv \Lted{u}_0 \in U$ and disturbance $\wdistnL\equiv 0\in U_d$ satisfies \eq{
\Lted{y}_n \to \Ltedop{P}(1)\Lted{u}_0 = \Pop \Lted{u}_0
}
as $n\to \infty$.
In terms of the original periodic system this means that the output $y(\cdot)$ on the interval $[n\persym,(n+1)\persym)$ corresponding to the $\persym$-periodic input $u(\cdot)$ 
such that $u(\cdot)=u_0(\cdot)\in \Lp[2](0,\persym;U_0)$ on $[0,\persym]$
converges to the function 
\ieq{
(\Pop u_0)(\cdot)
}
as $n\to \infty$. In particular, if we choose the $\persym$-periodic input $u^k(\cdot)$ in such a way that $u^k(\cdot)=\varphi_k(\cdot)$ on $[0,\persym]$ for $k\in\N$, then the corresponding output $y^k(\cdot)$ on the interval $[n\persym,(n+1)\persym)$ converges to 
\eq{
(\Pop\varphi_k)(\cdot) = \sum_{l=1}^\infty c_{lk} \psi_l(\cdot)
}
in the $\Lp[2]$-norm as $n\to\infty$. The coefficients $c_{kl}=\iprod{(\Pop\varphi_k)(\cdot)}{\psi_l}_Y$ can thus be approximated with $c_{kl}\approx\iprod{y^k(n\persym+\cdot)}{\psi_l}_Y$ for a sufficiently large $n$. For large $M,N\in\N$ the matrix 
\eq{
\Pop_{MN} = (c_{kl})_{kl}\in\C^{N\times M}
}
can then be used as an approximation of the operator
$\Pop$ from the subspace $\Span \set{\psi_k}_{k=1}^M\subset U$ to the subspace $\Span \set{\varphi_l}_{l=1}^N\subset Y$. In particular, the solution of the operator equation $y=\Pop u$ with $y=\sum_{l=1}^\infty y_l\psi_l$ can be approximated with
\eq{
u_{MN}(\cdot) = \sum_{k=1}^M u_k\varphi_k(\cdot)
}
where $(u_1,\ldots,u_M)^T = \Pop_{MN}\pinv (y_1,\ldots,y_N)^T$.

A similar procedure can be used to approximate $\Pdop\wdist^k$. Indeed, for $\persym$-periodic inputs $u(\cdot)\equiv 0$ and $w(\cdot)$ such that $w(\cdot)=\wdist^k(\cdot)$ on $[0,\persym]$ the corresponding output $y^k(\cdot)$ satisfies $\norm{y^k(n\persym+\cdot)-\Pdop \wdist^k(\cdot)}_{\Lp[2]}\to 0$ as $n\to \infty$. Therefore the function $(\Pdop \wdist^k)(\cdot)$ can be approximated with $y^k(\cdot)$ on the interval $[n\persym,(n+1)\persym)$ for a sufficiently large~$n$.

The bases of $U$ and $Y$ can be chosen freely. 
If $X$, $U_0$, and $Y_0$ are real spaces, it is convenient to use real bases of $U$ and $Y$, in which case $\Pop_{MN}\in \R^{N\times M}$.

\section{Controller Design for Coupled Harmonic Oscillators}
\label{sec:HarmOsc}

In this section we consider a system of harmonic oscillators with periodically time-varying damping and a one-sided time-dependent coupling. The full system is of the form
\eq{
\ddot{q}_1(t) + a_1(t) \dot{q}_1(t) + q_1(t) &= b(t)u(t) + w_d^1(t)\\
\ddot{q}_2(t) + a_2(t) \dot{q}_2(t) + q_2(t) &= g(t)q_1(t)+ w_d^2(t).
}
where $a_1(\cdot)$, $a_2(\cdot)$, $b(\cdot)$, $g(\cdot)$ are $2\pi$-periodic functions 
such that
\eq{
&a_1(t) = 1+\cos(2t), \qquad a_2(t) = 2-\frac{\abs{\pi-t}}{\pi},\\
 & b(t) = 1+ \frac{t(2\pi-t)}{\pi}, \qquad g(t) = 1+\frac{\sin(3t)}{4}
}
for $t\in[0,2\pi]$.
Our aim is to design a control input $u(t)$ in such a way that the measured position $y(t)=q_2(t)$ of the second oscillator tracks the reference signal $\yref(t)=1+\sin(t)$ despite the disturbance signals $w_d^1(\cdot)$ and $w_d^2(\cdot)$ that are linear combinations of the functions $\cos(2t)$ and $\sin(t)$.

The coupled harmonic oscillators can be written as a periodic system of the form~\eqref{eq:plantintro} on $X=\R^4$ and with $\persym=2\pi$, $U_0 = \R$, $Y_0=\R$, $U_{d0}=\R^2$,  $A(\cdot)\in C_\persym(\R,\R^{4\times 4})$, $B(\cdot)\in C_\persym(\R,\R^4)$,  $B_d(\cdot)\in C_\persym(\R,\R^{4\times 2})$, $C(\cdot)\equiv C\in \R^{1\times 4}$, and $D(\cdot)\equiv 0\in \R$. The system is exponentially stable since 
$\abs{\gl}<1$ for all $\gl\in\gs(U_A(\persym,0))$.

The operator $\Pop\in \Lin(U,Y)$  can be approximated with a matrix $\Pop_{MN}$ based on measurements from the system using the method in Section~\ref{sec:P1measurement}. In particular, we approximate the elements in the spaces $U=Y= \Lp[2](0,\persym;\R)$ with $21$ basis functions of the form $\varphi_k(\cdot)=\psi_k(\cdot)=\frac{1}{\sqrt{2\pi}}e^{ik\cdot}$ for $-10\leq k\leq 10$. The interval of the measurement was chosen to be $[n\persym,(n+1)\persym)$ for $n=10$.
Similarly, the functions $y_{d,k}=\Pdop \wdist^k$ can be approximated based on measurements for the inputs
$\wdist^1 =(\cos(2\cdot),0)^T$, $\wdist^2=(\sin(\cdot),0)^T$, $\wdist^3=(0,\cos(2\cdot))^T$, and $\wdist^4=(0,\sin(\cdot))^T$.

\subsection{Feedforward Control}

If the disturbance signal is completely known, we can use the static control law presented in Theorem~\ref{thm:ORPstatic} to solve the output tracking problem. Figure~\ref{fig:OscFFoutput} shows the output of the controlled system
for the disturbance signals $w_d^1(t)=0.4\cos(2t)+0.3\sin(t)$ and $w_d^2(t) = 0.2\cos(2t)+0.6\sin(t)$ and the initial state $x_0=0\in\R^4$ of the system.

\begin{figure}[h!]
  \begin{center}
\includegraphics[width=.68\linewidth]{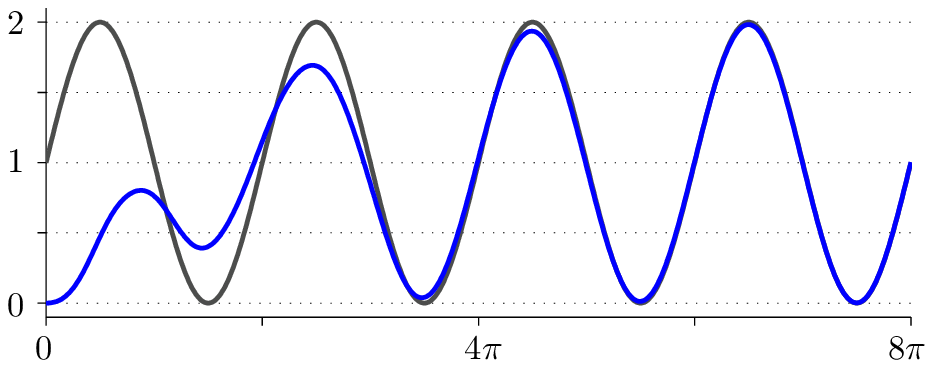}
  \end{center}
  \caption{Output (blue) with the feedforward control law.}
  \label{fig:OscFFoutput}
  \end{figure}

\subsection{Error Feedback Control}

If the amplitudes in the disturbance signals are unknown, the output regulation problem can be solved with a dynamic error feedback controller
 in Theorem~\ref{thm:ORPfeedback}. The controller is a $5$-dimensional discrete-time system 
\eq{
z_{n+1}=z_n - Q^\ast \pmat{\iprod{y(n\persym+\cdot)-\yref}{\yref}\\
( \iprod{y(n\persym+\cdot)-\yref}{y_{d,k}})_{k=1}^4
}
, \qquad n\geq 0
}
and the control input $u(\cdot)$ is determined by $u(n\persym+\cdot)=K_0z_n$ on the interval $[n\persym,(n+1)\persym)$ for all $n\geq 0$. As above, we approximate the functions 
$y_{d,k}(\cdot)=\Pdop \wdist^k$ 
and $\FFfun^k(\cdot)$ with $21$ basis functions $\varphi_l$ and $\psi_l$, respectively. We choose $K_0=[\FFfun^0,\ldots,\FFfun^4]\in \Lin(\C^5,U)$ and $Q=V\Lambda^{-1/2}$, where $V$ and $\Lambda$ are obtained from the  the singular value decomposition $V\Lambda V^\ast$ of the positive definite matrix $(\Pop K_0)^\ast \Pop K_0$. 

\begin{rem}
  \label{rem:Oscepschoice}
  The construction of the feedback controller requires $\eps>0$ to be chosen so that the closed-loop system is exponentially stable. The closed-loop stability can be tested for a given $\eps>0$ by simulating the original periodic system and the discrete-time controller on the interval $[0,2\pi]$ for initial states $x^e_0=(x_0,z_0)^T=\phi_k\in X\times Z=\C^{9}$, where $\phi_k$ are the Euclidean basis vectors. The final states $x^e_1=(x(2\pi),z_1)^T$ of the simulations are the corresponding columns of the closed-loop system matrix $A_e$ whose eigenvalues determine the closed-loop stability. These simulations can be used to optimize $\eps>0$ in such a way that the stability margin of the closed-loop system is sufficiently large while the imaginary parts of the eigenvalues of $A_e$ remain relatively small.
\end{rem}

Figures~\ref{fig:OscFBoutput} and~\ref{fig:OscFBerrors} show the output of the controlled system and the errors $\norm{y(n\persym +\cdot)-\yref(\cdot)}$ for the disturbance signals $w_d^1(t)=0.1\cos(2t)$ and $w_d^2(t) = 0.1\cos(2t)-0.1\sin(t)$ and initial state $x_0=0\in\R^4$. Using the procedure in Remark~\ref{rem:Oscepschoice} the parameter $\eps>0$ was chosen as $\eps = 0.25$, and we let $z_0=0\in \C^5$.

\begin{figure}[h!]
  \begin{center}
\includegraphics[width=.68\linewidth]{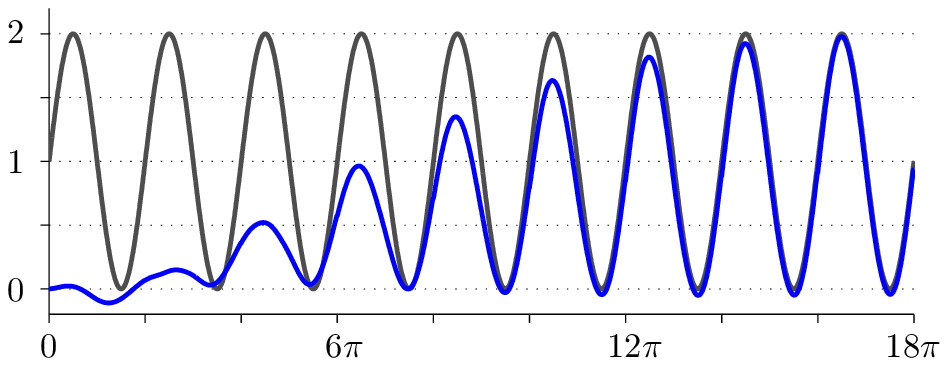}
  \end{center}
  \caption{Output (blue) with the feedback controller.}
  \label{fig:OscFBoutput}
\begin{center}
  \includegraphics[width=.68\linewidth]{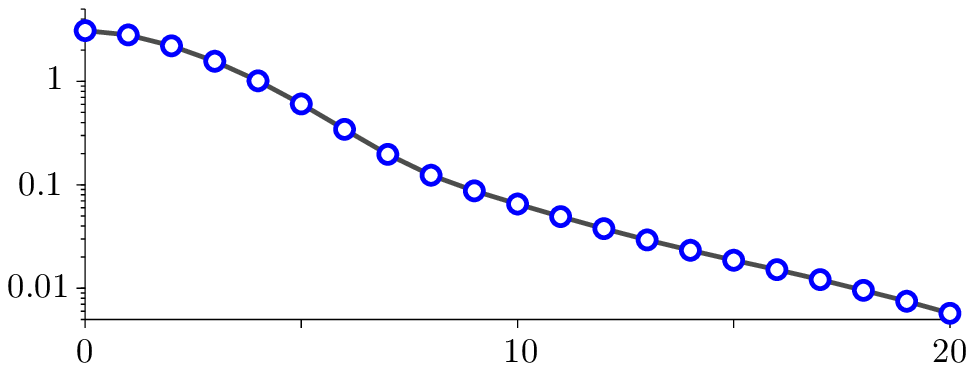}
  \end{center}
  \caption{Errors $\norm{y(n\persym+\cdot)-\yref}$ for $0\leq n\leq 20$ ($\log$-scale).
  }
  \label{fig:OscFBerrors}
\end{figure}

\subsection{Approximate Robust Control}

Finally, we construct a dynamic feedback controller in Theorem~\ref{thm:RORPapprox} to achieve approximate robust output tracking for the system of oscillators. 
If we choose $Y_N=\Span \set{\varphi_k}_{k=-7}^7$ where $\varphi_k = \frac{1}{\sqrt{2\pi}}e^{ik\cdot}$, then the constructed controller has dimension $r=\dim Y_N=15$. For the simulations we approximate the spaces $U=Y=\Lp[2](0,2\pi)$  with $\Span \set{\varphi_k}_{k=-14}^{14}$, and we denote by $\Pop_N$ the corresponding approximation of the operator $Q_N\Pop : U\to Y_N$.

The controller parameters were chosen so that $Z=\C^{15}$, $G_1 = I$, and $Q_N$ is the projection onto $Y_N$. In order to ensure that $\gs(G_{20}Q_N\Pop K_0)\subset \C_-$ we chose $G_{20}=\diag(\gs_1,\ldots,\gs_{15})\inv V_1^\ast$ and $K_0=-\tilde{V}_2$ where $V_1$, $V_2$, and $\Sigma=\diag(\gs_1,\ldots,\gs_{15})\in \R^{15\times 29}$ were obtained from the singular value decomposition $V_1\Lambda V_2^\ast$ of $\Pop_N\in \C^{15\times 29}$, and $\tilde{V}_2$ contains the first $r=15$ columns of $V_2$. Finally, using the procedure in Remark~\ref{rem:Oscepschoice} we chose $\eps=0.2$. Figures~\ref{fig:OscRoboutput} and~\ref{fig:OscRoberrors} show the behaviour of the output and the regulation error for a $2\pi$-periodic triangular reference signal and the disturbance signals $w_d^1(t)=0.3\sin(t)$ and $w_d^2(t)\equiv 0.2$, and for the intial states $x_0=0$ and $z_0=0$ of the system and the controller.

\begin{figure}[h!]
  \begin{center}
\includegraphics[width=.68\linewidth]{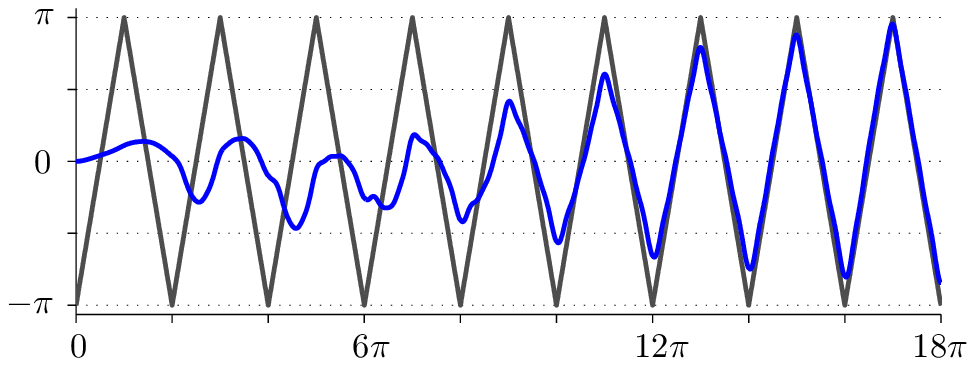}
  \end{center}
  \caption{Output (blue) with the feedback controller.}
  \label{fig:OscRoboutput}
\end{figure}
  \begin{figure}
\begin{center}
  \includegraphics[width=.68\linewidth]{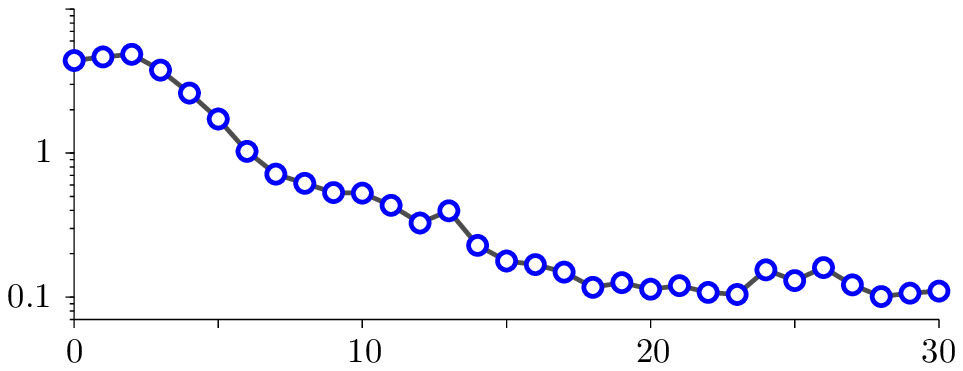}
  \end{center}
  \caption{Errors $\norm{y(n\persym+\cdot)-\yref}$ for 
$0\leq n\leq 30$ ($\log$-scale).
  }
  \label{fig:OscRoberrors}
\end{figure}

The asymptotic error estimate in Theorem~\ref{thm:RORPapprox} can be approximated numerically using the operator $\Pop$ measured from the system's response.  This way we can show that the regulation error is of order
\eq{
\norm{y(n\persym+\cdot)-\yref(\cdot)}_{\Lp[2]} \approx 0.1
}
as $n\to \infty$.

\vspace{-.2cm}

\section{Controller Design for a Periodic Heat Equation}
\label{sec:heatsys}

In this section we design controllers for a stable nonautonomous heat equation with boundary disturbances.
The system is 
determined by the
partial differential equation
  \eq{
  x_t(\xi,t)&= \frac{1}{6} \Delta x(\xi,t) + a(t)\chi_{\Omega_0}x(\xi,t) + 4\chi_{\Omega_1} u(t)\\[1ex]
  x(\xi,t)&=0 \quad 
  \forall \xi\in \partial\Omega\setminus \Gamma_0, \quad
  \pd{x}{n}(\xi,t)=\wdist(t) \quad 
  \forall \xi \in\Gamma_0\\
  y(t)&=4\int_{\Omega_2} x(\xi,t)d\xi 
  }
on the Hilbert space $X=\Lp[2](\Omega)$, where $\Omega=[0,1]\times [0,1]$, $\Omega_0=[0,1]\times [1/4,3/4]$, 
and $\Gamma_0 = \setm{(\xi_1,0)\in \Omega}{0\leq \xi_1\leq 1}$.
The control and observation are distributed over the regions $\Omega_1=[0,1/4]\times [0,1]$ and $\Omega_2=[3/4,1]\times [0,1]$, respectively, and $a(\cdot)$ is a $2\pi$-periodic function such that
\eq{
a(t) = 
\left\{
\begin{array}{cl}
  1&0\leq t<\pi\\
  3&\pi\leq t<3\pi/2\\
  2&3\pi/2\leq t<2\pi
\end{array}
\right.
}
for $t\in [0,2\pi]$.
The evolution family $\FM_A(t,s)$ is obtained as a composition of the strongly continuous semigroups $T_1(t)$, $T_2(t)$, and $T_3(t)$ generated by the operators $A_1=\frac{1}{6}\Delta+\chi_{\Omega_0}(\cdot)$, $A_2=\frac{1}{6}\Delta+3\chi_{\Omega_0}(\cdot)$, and $A_3=\frac{1}{6}\Delta+2\chi_{\Omega_0}(\cdot)$, respectively. The domains of the generators are
$\Dom(A_k)=\setm{x\in H^2}{x(\xi)=0 ~ \mbox{on} ~ \xi\in \partial\Omega\setminus \Gamma_0, ~ \pd{x}{\xi_2}(\xi)=0 ~ \mbox{on} ~ \xi\in  \Gamma_0}$
for $k=1,2,3$.
In particular, we have
\eq{
\FM_A(2\pi,0) 
= T_3(\pi/2)T_2(\pi/2)T_1(\pi).
}
The boundary disturbance corresponds to $B_d=\gd_{\Gamma_0}\in \Lin(\C,X_{-1})$. Since $B_d$ is admissible with respect to $A_k$ for all $k=1,2,3$, the operator $\Ltedop{B}_d$ is well-defined and bounded.

For simulations, the state of the heat equation was approximated with a finite difference scheme with 12 equally spaced points in both spatial dimensions. Precise characterization of the range of the operator $\Pop$ for the periodic heat system would be difficult, but it is immediate that all functions $y\in \ran(\Pop)$ must possess a certain level of smoothness. Therefore,
achieving exact output tracking of reference signals that are not continuously differentiable will be impossible in this example.

\subsection{Feedforward Control}

We begin by designing a control law to achieve output tracking of the $2\pi$-periodic reference signal $\yref(t) = -\frac{1}{3}\sin(3t)+\sin(t)$ despite the disturbance signal $\wdist(t) = 2\cos(2t)+3\sin(2t)$ on the boundary. Figure~\ref{fig:HeatFFoutput} shows the output of the controlled system for the initial state $x_0(\xi)\equiv -1$. The operator $\Pop$ and the function $\Pdop \wdist$ were approximated using measurements from the system on $[n\persym,(n+1)\persym)$ for $n=12$ using $\Span \set{\varphi_k}_{k=-10}^{10}$ with $\varphi_k = \frac{1}{\sqrt{2\pi}}e^{ik\cdot}$ as an approximation for the spaces $U=Y=\Lp[2](0,2\pi)$.

\begin{figure}[h!]
  \begin{center}
\includegraphics[width=.68\linewidth]{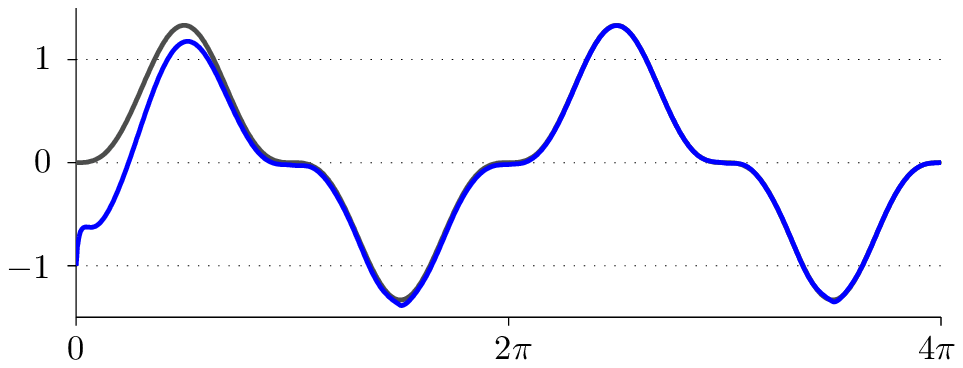}
  \end{center}
  \caption{Output (blue) with the feedforward control law.}
  \label{fig:HeatFFoutput}
  \end{figure}

\vspace{-.5cm}

\subsection{Approximate Robust Control}

We will now construct the controller in Theorem~\ref{thm:RORPapprox} to achieve approximate robust output regulation for the periodic heat equation. We choose $Y_N=\Span \set{\varphi_k}_{k=-7}^7$ where $\varphi_k = \frac{1}{\sqrt{2\pi}}e^{ik\cdot}$ and the resulting controller has dimension $r=\dim Y_N=15$. For the simulations we approximate $U=Y=\Lp[2](0,2\pi)$ with $\Span \set{\varphi_k}_{k=-14}^{14}$, and denote by $\Pop_N$ the corresponding approximation of $Q_N\Pop : U\to Y_N$.

For the controller we choose $Z=\C^{15}$, $G_1 = I$,
 and let $Q_N$ be the projection onto $Y_N$. 
 To achieve $\gs(G_{20}Q_N\Pop K_0)\subset \C_-$ we choose $G_{20}=\diag(\gs_1,\ldots,\gs_{15})\inv V_1^\ast$ and $K_0=-\tilde{V}_2$ where $V_1$, $V_2$, and $\Sigma=\diag(\gs_1,\ldots,\gs_{15})$ $\in \R^{15\times 29}$ are from the singular value decomposition $V_1\Lambda V_2^\ast$ of $\Pop_N\in \C^{15\times 29}$, and $\tilde{V}_2$ consists of the first $r=15$ columns of $V_2$. 
We used the procedure in Remark~\ref{rem:Oscepschoice} to choose $\eps=0.35$. Figures~\ref{fig:HeatRoboutput} and~\ref{fig:HeatRoberrors} show the behaviour of the output and the regulation error for a $2\pi$-periodic triangular reference signal and the disturbance $w_d(t)=0.3\sin(t)$, and for the intial states $x_0(\xi)\equiv 0$ and $z_0=0$ of the system and the controller.

\begin{figure}[h!]
  \begin{center}
\includegraphics[width=.68\linewidth]{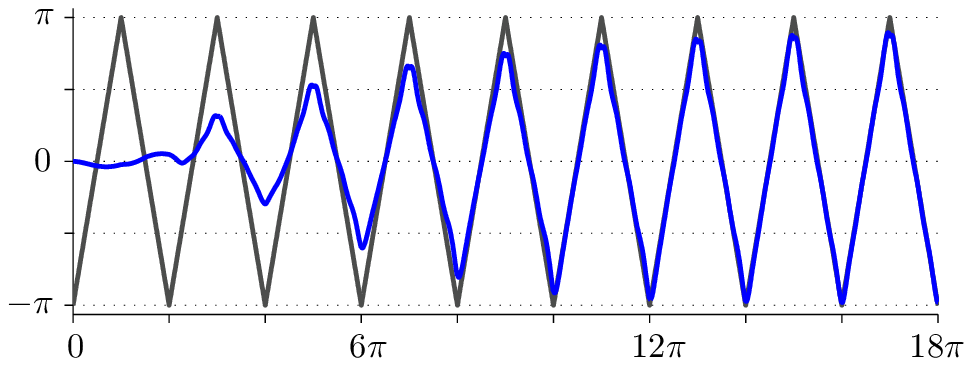}
  \end{center}
  \caption{Output (blue) with the feedback controller.}
  \label{fig:HeatRoboutput}
\begin{center}
  \includegraphics[width=.68\linewidth]{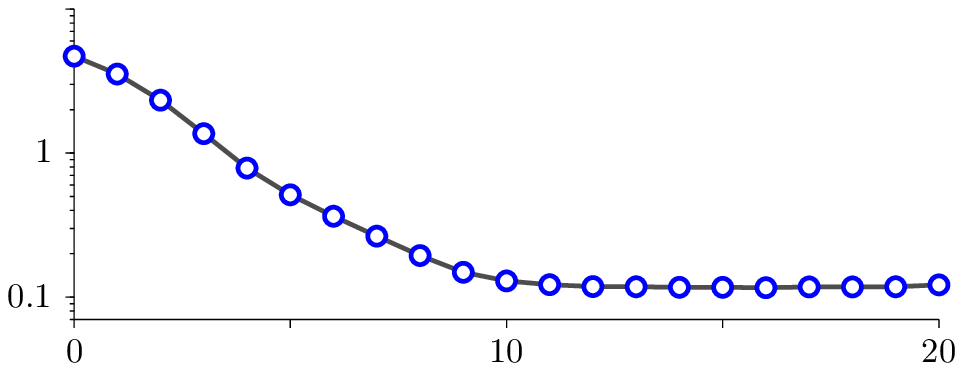}
  \end{center}
  \caption{Errors $\norm{y(n\persym+\cdot)-\yref}$ for 
$0\leq n\leq 20$ ($\log$-scale).
  }
  \label{fig:HeatRoberrors}
\end{figure}

The asymptotic error estimate in Theorem~\ref{thm:RORPapprox} can be approximated numerically using finite difference approximation and the operator $\Pop$ measured from simulations. Based on this approximation we get that the regulation error is of order
\eq{
\norm{y(n\persym+\cdot)-\yref(\cdot)}_{\Lp[2]} \approx 0.12
}
as $n\to \infty$.

\section{Conclusions}
\label{sec:conclusions}

In this paper we have studied the construction of controllers for output regulation and robust output regulation of continuous-time periodic systems. The constructions are based on expressing the original periodic system as an autonomous discrete-time system using the lifting technique. 
At the same time, the presented results also offer new methods for constructing controllers for output regulation of autonomous finite and infinite-dimensional systems in the situations where the signals $\yref(\cdot)$ and $\wdist(\cdot)$ are $\persym$-periodic functions.

Throughout the paper we have concentrated on the case where the reference and disturbance signals have the same period length $\persym>0$ as the system's parameters. The most important topic for future research is to extend the controller constructions for more general signals $\yref(\cdot)$ and $\wdist(\cdot)$ that are not periodic functions, or have different period lengths.

%\bibliographystyle{plain}
%\bibliography{../../../reference}

\end{document}